%% file: twoiterates.tex
\documentclass[a4paper,twoside,11pt]{amsart}
\pdfoutput=1

\input{defaultpreamble}

\newcommand{\averaged}{\mathscr{A}}

\newcommand{\leb}[1]{{\rm Leb}(#1)}
\newcommand{\prob}{\mathbb{P}}
\newcommand{\expec}[2]{\mathbb{E}_{#1}\left(#2\right)}

\title{Fermi Acceleration in anti-integrable limits of the standard map}
\author{Jacopo De Simoi}
\address{Jacopo De Simoi\\
  Dipartimento di Matematica\\
  II Universit\`{a} di Roma (Tor Vergata)\\
  Via della Ricerca Scientifica, 00133 Roma, Italy.}
\urladdr{\href{http://www.mat.uniroma2.it/~desimoi/redirect.php?refer=gamma2}{http://www.mat.uniroma2.it/~desimoi}}
\email{\href{mailto:desimoi@mat.uniroma2.it}{\texttt{desimoi@mat.uniroma2.it}}}

\begin{document}
\maketitle
\begin{abstract}
  We consider a dynamical system on the semi-infinite cylinder which models the high energy
  dynamics of a family of mechanical models.  We provide conditions under which we ensure
  that the set of orbits undergoing Fermi acceleration has measure zero.
\end{abstract}
\input{introduction}
\input{definitions}
\input{induceddynamics}
\input{criticalsets.definition}
\input{equidistribution}
\input{reduction}
\appendix
\input{criticalsets.estimates}
\input{bibliography}
\end{document}

%% file: defaultpreamble.tex
\usepackage{amsfonts}
\usepackage{amsmath}
\usepackage{amssymb}
\usepackage{amsthm}
\usepackage{mathrsfs}
\usepackage{eucal}
\usepackage{verbatim}
\usepackage{enumerate}
\usepackage{graphicx}
\usepackage{color}

\ifx\XeTeXversion\undefined
\usepackage[english]{babel}
\else
\usepackage{microtype}
\usepackage{fontspec}
\usepackage{polyglossia}
\setmainlanguage{english}
\fi

\numberwithin{equation}{section}

\newcommand{\naturals}{\mathbb{N}}
\newcommand{\integers}{\mathbb{Z}}
\newcommand{\reals}{\mathbb{R}}

\newcommand{\sone}{\mathbb{T}^1}
\newcommand{\torus}{\mathbb{T}}

\newcommand{\fa}{\forall\,}
\newcommand{\ex}{\exists\,}
\newcommand{\eps}{\varepsilon}
\newcommand{\defeq}{\doteqdot}
\newcommand{\deh}{\textup{d}}
\newcommand{\dpar}[2]{\frac{\partial #1}{\partial #2}}
\newcommand{\de}[2]{\frac{\textup{d} #1}{\textup{d} #2}}
\newcommand{\st}{\textrm{ s.t. }}
\newcommand{\bigo}[1]{\mathcal{O}(#1)}

\newcommand{\const}{C_\#}
\newcommand{\intr}{\textup{int }}
\newcommand{\inv}{^{-1}}

\newtheorem{rmk}{Remark}[section]
\newtheorem{thm}[rmk]{Theorem}
\newtheorem{lemma}[rmk]{Lemma}
\newtheorem{prp}[rmk]{Proposition}
\newtheorem{cor}[rmk]{Corollary}
\newtheorem{definition}[rmk]{Definition}

\newtheorem{conj}[rmk]{Conjecture}

\newtheorem*{mthm}{Main Theorem}

\newcommand{\eref}[1]{(\ref{#1})}

%% file: introduction.tex
\section{Introduction and results}
In this work we study dynamical properties of a family of exact area-preserving twist maps
on the semi-infinite cylinder. This family describes an approximation of the high energy
dynamics of a class of generalizations \cite{Pust1,Pust2,Ortega1,Ortega2,Dima} of the
Fermi-Ulam ping-pong \cite{Ulam}. Among other examples, the dynamics of some $n$-body
problems, such as the Sitnikov three-body problem (see \cite{Si,GK11}) can also be
described by means of maps which are similar to the ones we consider in this work.

One of the most remarkable differences between finite and infinite measure dynamical
systems is that the latter are possibly lacking the \emph{recurrence} property, which is
guaranteed by Poincar\'e theorem in the former situation. If the set of \emph{wandering}
(that is, non-recurrent) points has positive measure, we say that the system is
\emph{dissipative}, otherwise we say it is \emph{conservative}. Conservativity is a very
desirable property for infinite measure systems and it is the starting point to prove (and
even to define in a satisfactory manner) ergodic and statistical properties such as bounds
for the decay of correlations (see for instance \cite{Lenci}). In our framework,
conservativity has also a very concrete physical interpretation for the mechanical systems
which our maps relate to; in our case, in fact, if a point belongs to the wandering set,
then, necessarily, its energy will tend to infinity with time; they are in fact orbits
which undergo so-called \emph{Fermi acceleration} (see \cite{F49,F54}).  In this paper we
show that, if suitable conditions of the parameters are satisfied, the maps of our family
are conservative; consequently, the mechanical systems which can be modeled by the above
choice of parameters, allow only a null set of Fermi accelerated orbits.

Another quite interesting feature of the family under consideration is its affinity with
the Chirikov-Taylor standard map (see \cite{Chirikov}). For instance (see e.g.\
\cite{Pust1}) the Fermi-Ulam ping-pong system is well described, for large values of the
non-cyclic variable (i.e.\ for high energies), by the dynamics of the standard map for
small coupling parameter, i.e. in the quasi-integrable regime. On the other hand, the
family of maps under our consideration are such that, for large values of the non-cyclic
variable, the dynamics is described by the standard map for large coupling
parameters, that is, far away from the integrable regime. In this sense our maps can be
regarded as anti-integrable limits of the standard map. For this reason, we believe that
our family shares part of the ``universality'' properties of the standard map and that
this study can indeed be useful for a variety of different situations. Additionally, most
of the difficulties we will encounter in our work will be directly related to
corresponding issues for the standard map and we can expect the techniques employed in the
present work to be successfully applied also to that more challenging study.

%% file: definitions.tex
\input{definitions.p} Let $\sone\defeq\reals/2\pi\integers$ and
$\phspace\defeq\sone\times\reals_{+}$ be the semi-infinite cylinder. Let $\phi$ be a
smooth real-valued function on $\sone$; for definiteness we assume
\begin{equation}
  \phi(\theta)=\sin(\theta),
  \label{e_definitionPhi}
\end{equation}
and for $A>0$ we let $\phi_A\defeq A \phi$. For $\hat\T,\gamma\in(0,\infty)$ fix $\T_\gamma$ to
be a smooth orientation-preserving diffeomorphism of $\reals_{+}$ given by:
\begin{equation}
  \T_\gamma(\cv{})=\hat{\T}\cdot \cv{}^\gamma.\label{e_definitionY}
\end{equation}
Finally, let $F_{A,\gamma}:\phspace\to\phspace$ be the area-preserving map given by the following formula:
\begin{equation}
F_{A,\gamma}:\left(\begin{array}{c}\ct{}\\\cv{}\end{array}\right)
\mapsto
\left(\begin{array}{c}\ct{}+\T_\gamma(\cv{})\\\cv{}+2\dot\phi_A(\ct{}+\T_\gamma(\cv{}))\end{array}\right),
\label{e_mainDefinition}
\end{equation}
if $y>L$ for some $L$ large enough, and continued to a smooth map on $0\leq y \leq L$; the
exact form of the continuation is irrelevant for our statements, since it only influences
the dynamics on a compact region of the phase space. We can always assume $\hat{\T}=1$,
otherwise we let $\cv{}\mapsto\hat{\T}^{1/\gamma}\cv{}$ and
$A\mapsto\hat{\T}^{1/\gamma}A$. We fix once and for all the values of $A$ and $\gamma$ and
then study the dynamics of $F_{A,\gamma}$; for sake of simplicity we will therefore drop
all subscripts $A$ and $\gamma$ appearing in the definitions since this will not be source
of confusion.  Furthermore, introduce the convenient notation $(\coo{k})=F^k(\coo{0})$; we
define the \emph{escaping set} $\escaping$ as follows:
  \[
  \escaping\defeq\{(\coo{0})\st \lim_{n\to\infty} \cv{n}=\infty \}.
  \]
In our setting the escaping set coincides with the wandering set; however, since we find the word
``escaping'' more descriptive, we prefer it over our other option.  Our interest is to
provide results on the largeness of the set $\escaping$ depending on the values of $A$ and
$\gamma$.

The map given by \eref{e_mainDefinition} can be obtained (see \cite{ghianda} for a
detailed derivation) as an high energy approximation (the so called \emph{static wall
  approximation}) of a suitable Poincar\'e map of the dynamics of a particle bouncing on a
periodically oscillating infinitely heavy plate while subject to a potential force given
by a power law with exponent $2/(\gamma+1)$. In this model the $x$ variable corresponds to
the time of a collision with the moving plate, while $y$ corresponds to the
post-collisional velocity; some results, which have been originally obtained for the
mechanical problem, can indeed be adapted to our situation; we list a selection of them,
which are directly related to this work.
\begin{thm}[Pustylnikov \cite{Pust1}]\label{t_pust}
  If $\gamma = 1$, then the set of escaping orbits contains an open set for an open set of
  values of $A$.
\end{thm}
\begin{thm}[Ortega \cite{Ortega1}]\label{t_ort}
  If $\gamma=0$ and certain resonance conditions are satisfied, then the set of escaping orbits $\escaping$ contains an open set of the phase space.
\end{thm}
\begin{thm}[Dolgopyat \cite{Dima}]\label{t_dimaKAM}
  If $\gamma\in(0,1)$, then the set of escaping orbits $\escaping$ is empty.
\end{thm}
The proof of Theorem~\ref{t_pust} is indeed quite simple in our case; in fact, if
$\gamma=1$, then the map $F$ for large $y$ is given by the unfolding of the standard map
on the cylinder. In this case it is easy to prove that, for an open set of parameters, we
can find a periodic orbit on $\mathbb{T}^2$ given by centers of a chain of elliptic
islands of period $N$ for the standard map on the torus $\mathbb{T}^2$ which lifts to a
non-periodic orbit on the cylinder such that $F^{N}:(\coo{})\mapsto(\ct{},\cv{}+\nu)$,
where $\nu$ is a positive integer. This implies that the lift of any elliptic island in
the chain is a subset of the escaping set, which consequently has positive, hence
infinite, measure.  The statement of Theorem~\ref{t_ort} is also not surprising, in fact
if $\gamma=0$, then the function $\T$ is constant; then if we have a resonance condition
between $\T$ and the period of $\phi$, it is plausible that escaping orbits can indeed
arise. Finally, Theorem~\ref{t_dimaKAM} follows from showing the existence of KAM tori for
large values of $y$; in fact their presence prevents diffusion and hence implies the
result.  On the one hand, for values of $\gamma$ greater than $1$, the set of escaping
orbits $\escaping$ is non-empty; in fact it was proved in \cite{ghianda} that the set has
full Hausdorff dimension.  On the other hand, the following result holds:
\begin{thm}[Dolgopyat \cite{Dima}]\label{t_dima}
  If $\gamma>5$, then the set of escaping orbits $\escaping$ has zero measure.
\end{thm}
In the same article it is indeed conjectured that
\begin{conj}\label{conj_dima}
  If $\gamma>1$, then the set of escaping orbits $\escaping$ has zero measure.
\end{conj}
The main result of this work is an improvement of Theorem~\ref{t_dima} and constitutes a
step towards the proof of the conjecture.
\begin{mthm}\label{t_main}
  If $\gamma>2$, then the set of escaping orbits $\escaping$ has zero measure.
\end{mthm}
As it will be clear once the proof will be explained, our strategy for the proof does not
work if $\gamma\leq 2$. By performing some numerical simulations, it seems likely that
this is not merely a technical problem; indeed, proving our Main Theorem for
the value $\gamma=2$ seems to require a somewhat different approach, and it is still out
of reach at the moment.

Additionally, the techniques developed in order to prove our Main Theorem are of
independent interest for the study of statistical properties of non-uniformly hyperbolic
systems which appear to present coexistence of elliptic and hyperbolic behavior; as will
be made clear later, one can regard the condition on $\gamma$ as a stipulation on how
fast the expansion rate of the map along the cyclic coordinate $x$ can grow along with the
non-cyclic coordinate $y$: if the growth rate is strong enough, then we can conclude that
the system is conservative.  Indeed, a part from the classical example of the standard
map, we believe that our techniques could be employed in more general analyses (e.g. the
one provided in \cite{DimaRR}).
\subsection*{Acknowledgments}
This paper constitutes a natural evolution of the most substantial part of my
Ph.D. dissertation. I am thus most thankful to my thesis advisor Prof. Dmitry Dolgopyat
for his strong and constant encouragement before and after my defense.  This work has been
partially supported by the European Advanced Grant Macroscopic Laws and Dynamical Systems
(MALADY) (ERC AdG 246953) and by the Fondation Sciences Math\'ematiques de Paris.


%% file: induceddynamics.tex
\input{definitions.p}

\section{Proof of our Main Theorem }

The set $\escaping$ is invariant, and thus, by Poincar\'e recurrence Theorem, has
either zero or infinite measure. Additionally, if $(\coo{})\in\escaping$, then for any
$y_*>0$ there exists a $k_*$ such that if $k>k_*,\,y_k\geq y_*$, therefore if we let
\[
\escaping_*=\{(\coo{0})\in\escaping\st y_k\geq y_*\text{ for }k\geq 0\}
\]
we conclude that $\escaping=\bigcup_{k\geq0} F^{-k}\escaping_*$, thus, in order to prove
our Main Theorem, it suffices to prove that there exists a $\cv{*}>0$ such that
$\escaping_*$ has zero measure.  It is convenient to introduce the notation
$\phspace_*=\{y\geq y_*\}$; we will need to choose $y_*$ very large in order to satisfy a
number of requirements, which will be stated in due course; in particular we will always
assume $y_*>L$ where $L$ is the constant introduced in the previous section.

The proof is based on the idea used for the proof of Theorem~\ref{t_dima}, however some
substantial improvements are required.  The proof of Theorem~\ref{t_dima} is based on
equidistribution on $\sone$ of $\ct{}$-components of almost all trajectories which do not
leave $\phspace_*$ in the future; this ultimately follows from proving that the expansion
along so-called \emph{standard curves} provides enough uniformity to deliver
equidistribution outside of a \emph{critical set} $\critical{}$ which has finite measure
for large enough $\gamma$. On the other hand, on $\critical{}$ the dynamics is recurrent
by Poincar\'e's Theorem and, therefore, orbits which return to $\critical{}$ infinitely many
times cannot escape. Equidistribution on the complement of $\critical{}$ allows then to
set up a comparison with a random walk which ultimately is used to prove that almost every
trajectory will eventually land in $\critical{}$.  By sharpening some estimates, we can
push this strategy to work up to $\gamma>3$ but it ultimately fails for any smaller
$\gamma$, since the measure of the critical set will necessarily be infinite in this
case. To obtain the result for $\gamma>2$, we do need to study the dynamics of the map
inside the critical set. In particular, the main strategy to improve the condition on
$\gamma$ is to recover some hyperbolicity of the dynamics inside $\critical{}$ by
considering successive iterations of the map; by doing so we obtain a smaller critical set
$\critical{*}\subset\critical{}$ whose measure will be finite also for smaller values of
$\gamma$. In \cite{ghianda} we proved that the measure of elliptic islands in the critical
set is infinite if $\gamma\leq4/3$; this implies that the above strategy cannot be used to
fully prove Conjecture~\ref{conj_dima}. The present work shows, however, that the strategy
can be successfully employed to prove our Main Theorem, for which it is sufficient to
consider a critical set obtained with a two-iterate scheme (critical set of order $2$);
however, in order to obtain the necessary equidistribution estimates, one needs to
consider several iterates; the number of iterates in fact tends to $\infty$ as $\gamma\to
2$.  Thus, in a sense, our strategy is optimal when using critical sets of order $2$.
\subsection{Standard pairs}
We will study equidistribution properties of the dynamics employing the technique of
\emph{standard pairs}. A curve $\Gamma\subset\phspace_*$ is said to be a \emph{basic
  curve} if it is a graph of a smooth function $\psi:I\to\reals$ where $I\subset\sone$ is
an interval. A \emph{basic pair} $\ell$ is then given by a basic curve $\Gamma_\ell$ and a
strictly positive smooth probability density $\rho_\ell$ on $I$: we write
$\ell=(\Gamma_\ell,\rho_\ell)$ where $\Gamma_\ell=(\ct{},\psi_\ell(\ct{}))$ for $\ct{}\in
I_\ell$. A basic pair defines a measure as follows: for any real valued Borel measurable
function $\averaged(\coo{})$ we define:
\[
\expec{\ell}{\averaged}\defeq\int_{\Gamma_\ell}\averaged\rho_\ell\,\deh\ct{} = \int_{I_\ell}\averaged(\ct{},\psi_\ell(\ct{}))\rho_\ell(\ct{})\deh \ct{},
\]
and for any Borel measurable set $E$:
\[
\prob_\ell(E)\defeq\expec{\ell}{1_{E}}.
\]
We introduce, for convenience, the function $\T_\ell(\ct{})=\T(\psi_\ell(\ct{}))$ and similarly $\T'_\ell(x)=\T'(\psi_\ell(\ct{}))$ and $\T''_\ell=\T''(\psi_\ell(\ct{}))$.  We denote by $\slope{\ell}$ the \emph{slope} of the basic curve $\Gamma_\ell$, i.e. $\slope{\ell}(\ct{})\defeq\dot\psi_\ell(\ct{})$, where the dot denotes differentiation with respect to the variable $\ct{}$. It is also useful to introduce the \emph{adapted slope} function $\tslope{\ell}(\ct{})\defeq\slope{\ell}(\ct{})+1/\T'_\ell(\ct{})$ and the \emph{local expansion rate} $\LL_\ell(\ct{})=\tslope{\ell}(\ct{})\T'_\ell(\ct{})$. Notice that definition \eqref{e_mainDefinition} implies that, if $(\coo{})\in\Gamma_\ell$ and $(\coop{})=F(\coo{})$:
\[
\LL_\ell(\ct{})=\left.\de{\ct{}'}{\ct{}}\right|_{\Gamma_\ell}(\ct{}).
\]
We denote by $\rr{\ell}(\ct{})\defeq\rho_\ell^{-1}(\ct{})\dot\rho_\ell(\ct{})$ the logarithmic derivative of $\rho_\ell$ . Finally, we define:
\[
\V_\ell\defeq\inf_{\ct{}\in I} \psi_\ell(\ct{}).
\]
\begin{lemma}\label{l_basicCurveIteration}
  Let $\ct{*}\in I_\ell$ such that $\tslope{\ell}(\ct{*})\not = 0$; then there exists $U\subset I_\ell$ a neigborhood of $\ct{*}$ such that:
  \begin{itemize}
  \item the curve $\Gamma'=F\Gamma|_U$ is the graph of a smooth function $\psi':I'\to\reals$;
  \item the pushforward $\rho'(\ct{}') = c'\rho(\ct{}(\ct{}'))/\LL_\ell(\ct{}(\ct{}'))$, where $\ct{}(\ct{}')=\pi_1 F^{-1}\Gamma'(\ct{}')$ and $c'$ is a normalizing constant, is a strictly positive smooth probability density on $I'$;
  \end{itemize}
  hence $\ell'=(\Gamma',\psi')$ is a basic pair. Moreover:
  \begin{subequations}\label{e_induced}
    \begin{align}
      \slope{\ell'}&= 2\ddot\phi(\ct{}')+\frac{1}{\T'_\ell}\left(1-\frac{1}{\LL_\ell}\right)\\
      \dslope{\ell'}&= 2\dddot\phi(\ct{}')+\frac{\dslope{\ell}}{\LL^3_\ell}-\frac{\T''_\ell}{\T'^3_\ell}\left(1-\frac{1}{\LL_\ell}\right)^3\\
      r_{\ell'}&= \frac{r_\ell}{\LL_\ell} - \frac{\dslope{\ell} \T'_\ell}{\LL^2_\ell}-\frac{\T''_\ell}{\T'^2_\ell}\left(1-\frac{1}{\LL_\ell}\right)^2,
    \end{align}
  \end{subequations}
  where all functions with subscript $\ell$ are evaluated at the point $\ct{}$ and all functions with subscript $\ell'$ are evaluated at the corresponding point $\ct{}'$.
\end{lemma}
\begin{proof}
  Equations \eqref{e_induced} immediately follow from the definitions assuming $\LL_\ell\not = 0$; on the other hand, if $\tslope{\ell}(\ct{*})\not = 0$ we know that there necessarily exists a neighborhood $U$ such that $\tslope{\ell}(U)\not\ni 0$. Therefore, (\ref{e_induced}b) implies that the curve $\Gamma'$ is a graph of a smooth function and $\LL_\ell\not = 0$ implies that $\rho'$ is strictly positive.

  Notice moreover that even if $\rho'$ depends on the choice of $U$, equations \eqref{e_induced} are well-defined and independent of $U$.
\end{proof}

We will shortly introduce the notion of \emph{standard pairs}, which are given by a special class of basic pairs. First, define a class of basic pairs that we call \emph{reference pairs}: geometrically, reference pairs are given by pieces of the image of a vertical line which are not too short nor too long endowed with a uniform density. Standard pairs will in turn be defined as being appropriately close to reference pairs.

Fix once and for all a sufficiently small $\delta>0$; we require $\delta$ to be smaller than the minimum distance between two consecutive critical points of $\dot\phi$; in our case it suffices to take $\delta<\pi/4$. We say that an interval $I\subset\torus^1$ is a \emph{standard interval} if $\delta/4<|I|<\delta$.
\begin{definition}
  A basic curve $\Gamma=(x,\psi(x))$ with $\psi:I\to\reals$ is said to be a \emph{reference curve} if $I$ is a standard interval and
  \begin{align*}
    \psi_\ell(\ct{}) &= 2\dot\phi(\ct{})+\T^{-1}\left(c+\ct{}\right)
  \end{align*}
  for some $c>0$; a basic pair $\ell$ is said to be a \emph{reference pair} if $\Gamma_\ell$ is a reference curve and $\rho_\ell\equiv|I_\ell|\inv$.
\end{definition}
Define the following functions:
\begin{align}\label{e_referenceSection}
  \slope{1}(\coo{0})&=2\ddot\phi(\ct{0})+\frac{1}{\T'(\cv{-1})}&
  \dslope{1}(\coo{0})&=2\dddot\phi(\ct{0})-\frac{\T''(\cv{-1})}{\T'^3(\cv{-1})}.
\end{align}
Then if $\ell$ is a reference pair, we have:
\begin{align*}
  \slope{\ell}(\ct{})&=\slope{1}(\ct{},\psi_\ell(\ct{}))&
  \dslope{\ell}(\ct{})&=\dslope{1}(\ct{},\psi_\ell(\ct{})).
\end{align*}
It is also convenient to define the function $\tslope{1}(\coo{})=\slope{1}(\coo{})+1/\T'(\cv{})$; we will always require $\cv{*}$ to be so large that $\|\slope{1}\|_{\phspace_*}<3A$ and $\|\tslope{1}\|_{\phspace_*}<3A$.
\begin{definition}
Let $I\subset \torus^1$ be an interval and $\rho$ a probability density on $I$; we say
that $\rho$ is regular if $r(x)=\rho^{-1}(x)\dot\rho(x)$ satisfies $\|r\|_I<1$.
\end{definition}
\begin{lemma}\label{l_comparison}
  Let $I$ be a standard interval; then there exist $0<\mu_1<\mu_2$ such that if $\rho$ is
  a regular probability density on $I$, then for any measurable set $E\subset I$ we have
  \[
  \mu_1\leb{E}<\prob(E)<\mu_2\leb{E}
  \]
\end{lemma}
\begin{proof}
By the regularity condition $|r|<1$, applying Gr\"onwall lemma to $\rho$ we obtain, for every $\ct{},\bar{\ct{}}\in I$:
\[
\rho(\bar{\ct{}})e^{-|\ct{}-\bar{\ct{}}|}\leq\rho(\ct{})\leq\rho(\bar{\ct{}})e^{|\ct{}-\bar{\ct{}}|};
\]
by taking $\bar{\ct{}}$ such that $\rho(\bar{\ct{}})=\bar\rho$ the average density and by the definition of standard interval we obtain:
\[
\mu_1 = \delta^{-1}e^{-\delta}<\rho(\ct{})< 4\delta^{-1}e^{\delta} = \mu_2\qedhere
\]
\end{proof}
\begin{definition}\label{d_definitionStandardPair}
  Fix $D$ a constant to be defined later; let $\ell$ be a basic pair and define
  $\Delta\slope{\ell}\defeq\slope{\ell}(\ct{})-\slope{1}(\ct{},\psi_\ell(\ct{}))$ and
  correspondingly
  $\Delta\dslope{\ell}\defeq\dslope{\ell}(\ct{})-\dslope{1}(\ct{},\psi_\ell(\ct{}))$. Then
  $\ell$ is said to be a \emph{standard pair} if $I_\ell$ is a standard interval,
  $\rho_\ell$ is regular and $\Gamma_\ell$ is locally close to a reference curve in the
  following sense:
  \begin{subequations}\label{e_definitionStandardPair}
  \begin{align}
    |\Delta\slope{\ell}(\ct{})|&<D\inv\T'_\ell(\ct{})^{-3/2}\label{e_standardPairSlope}\\
    |\Delta\dslope{\ell}(\ct{})|&< A/10\label{e_standardPairDslope}
  \end{align}
\end{subequations}
\end{definition}
The next lemma ensures that standard curves are globally close to reference curves.
\begin{lemma}\label{l_refclose}
  Let $\ell$ be a basic pair and $\bar\ell$ a reference pair such that $I_\ell=I_{\bar\ell}=I$; assume there exists a $\ct{*}\in I$ such that $\psi_\ell(\ct{*})=\psi_{\bar\ell}({\ct{*}})$. Then:
\[
\fa \ct{}\in I\quad |\psi_\ell(\ct{})-\psi_{\bar\ell}(\ct{})|<2\|\Delta\slope{\ell}\| |\ct{}-\ct{*}|.
\]
\end{lemma}
\begin{proof}
  Let $\V=\min\{\V_\ell,\V_{\bar\ell}\}$ and let
  \[
  \mu=\sup_{x\in\torus^1}\left|\dpar{\slope{1}}{\cv{}}(x,\V)\right|;
  \]
  it is immediate to check that for all $x\in\torus^1,\,\cv{}\geq\V$ we have $\left|\dpar{\slope{1}}{\cv{}}(\coo{})\right|\leq \mu= o(\V\inv)$; therefore we can write:
\begin{align*}
\left|\de{}{\ct{}}\left(\psi_\ell(\ct{})-\psi_{\bar\ell}(\ct{})\right)\right| &\leq \left|\slope{\ell}(\ct{}) - \slope{1}(\ct{},\psi_\ell(\ct{}))\right| + \left|\slope{1}(\ct{},\psi_\ell(\ct{})) - \slope{1}(\ct{},\psi_{\bar\ell}(\ct{}))\right|\\
&\leq \|\Delta\slope{\ell}\| + \mu|\psi_\ell(\ct{})-\psi_{\bar\ell}(\ct{})|.
\end{align*}
Let $J\subset I$ be the connected component of the set $\{|\psi_\ell(\ct{})-\psi_{\bar\ell}(\ct{})|<2\|\Delta\slope{\ell}\|\}$ containing $\ct{*}$; for all $\ct{}\in J$ and for large enough $\V$ we have:
\[
\left|\de{}{\ct{}}\left(\psi_\ell(\ct{})-\psi_{\bar\ell}(\ct{})\right)\right| \leq (1+2\mu)\|\Delta\slope{\ell}\|\leq 2\|\Delta\slope{\ell}\|,
\]
which in particular implies that $J=I$ and concludes the proof.
\end{proof}
Fix $K$ large and an interval $I\in\torus^{1}$ and let $S_I\in\phspace$ be the half-strip given by $I\times[K,\infty)$. We define \emph{adapted coordinates} on $S$ by straightening the foliation of $S_I$ given by reference curves. More precisely:
\begin{definition}
Fix $\bar{\ct{}}\in I$ and let
\begin{align*}
\kappa: \tilde{I}\times\reals^{+}&\to I\times\reals\\
        (\xi,\eta)&\mapsto\left(\xi+\bar{\ct{}},\psi_\eta(\xi+\bar{\ct{}})\right)
\end{align*}
where $\psi_\eta$ is a reference curve such that $\psi_\eta(\bar{\ct{}})=\eta$.
We define \emph{adapted coordinates} on $S_I$ by taking the restriction of $\kappa$ on $\kappa^{-1}S_I$
\end{definition}


%% file: criticalsets.definition.tex
\input{definitions.p}

\subsection{Critical sets}
We need to establish results regarding invariance properties of standard pairs; in order
to do so we need to obtain good geometrical and regularity bounds (to control $\slope{}$,
$\dslope{}$ and $\rr{}$) for the map $F$. Such bounds cannot be established everywhere;
points where this is not possible will belong to sets that we will call \emph{critical
  set}s. The definition of the critical sets depends on our requirements for a ``good''
bound, and therefore it is far from being unique. However, all critical sets need to
satisfy the following condition: every orbit that never visits the critical sets is
hyperbolic.
\begin{definition}
  Fix $K_1$, $K_2$ large; %
  we define $\critical{1}$ the \emph{critical set of order 1} and $\critical{2}$ the \emph{critical set of order 2} as follows:
  \begin{align*}
    \critical{1}&\defeq \left\{(\coo{0})\in\phspace_* \st |\tilde h_1(\coo{0})|<K_1 \T'(\cv{0})^{-1/2}\right\};\\
    \critical{2}&\defeq \left\{(\coo{0})\in\phspace_* \st |\tilde h_1(\coo{0})\tilde h_1(\coo{1})|<K_2 \T'(\cv{0})^{-1}\right\}\cap\critical{1};
  \end{align*}
\end{definition}%
Take $\bar K_2>4$ and define the set:
\[
  \corecritical{2}\defeq\left\{(\coo{0})\in\phspace_* \st |\tslope{1}(\coo{0})|<\bar{K}_2 \T'(\cv{0})^{-1}\right\}.
\]
We choose $K_2$ so large that $\corecritical{2}\subset \critical{2}$; the set $\corecritical{2}$ will be called the \emph{core} of the critical set $\critical{2}$. We furthermore assume $\cv{*}$ to be large enough so that $\{\ddot\phi(\ct{})=0\}\subset\corecritical{2}$.
Notice moreover that:
  \begin{equation*}
    \T'(\cv{k})=\T'(\cv{0})\left(1+\bigo{|k|\cv{0}^{-1}}\right)
  \end{equation*}
  which yields, for any given $k$:
  \begin{equation} \label{e_confuse}
    \fa\eps>0\ \ex\bar{\cv{}}\st\cv{0}>\bar{\cv{}}\Rightarrow(1-\eps)\T'(\cv{0})<\T'(\cv{k})<(1+\eps)\T'(\cv{0}).
  \end{equation}
  Thus we can choose $K_1$ large enough to ensure that $\critical{1}\cap F^{-1}\critical{1}\subset\critical{2}$.
We now proceed to define the \emph{augmented} critical sets, which are suitably defined neighborhood of the critical sets. Fix $\hat K_1>K_1$ to be determined later and define the following set:
\begin{equation}
  \hcritical{1}\defeq\{(\coo{0})\in\phspace_*\st |\tslope{1}(\coo{0})|<\hat K_1\T'( \cv{0})^{-1/2}\}.
  \label{e1_invc}
\end{equation}
We now extend $\critical{2}$ to $\hcritical{1}$:
\[
\critical{2}^{*}\defeq\left\{(\coo{0})\in\phspace_*\st |\tilde h_1(\coo{0})\tilde h_1(\coo{1})|<K_2 \T'(\cv{0})\inv\right\}\cap\hcritical{1};
\]
we furthermore require $K_2$ to be so large that the inclusion
$F^{-1}\critical{1}\cap\hcritical{1}\subset\critical{2}^*$ holds.  Then, fix $
\hat K_2>K_2$ also to be determined later and define:
\begin{align*}
  \hcritical{2}\defeq\{(\coo{})\in\phspace_* \st  |\tilde h_1(\coo{0})\tilde
  h_1(\coo{1})|<\hat K_2 \T'(\cv{0})\inv\}\cap \hcritical{1}
\end{align*}
We describe in a lemma the geometrical features of critical sets, which are sketched in
Figure~\ref{f_criticals}.
\begin{figure}[!ht]
  \def\svgwidth{10cm}
  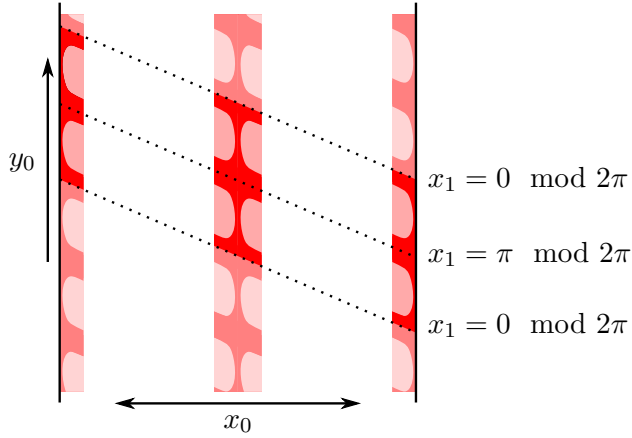
  \caption{Sketch of the geometry of $\critical1$ and $\critical2$ for large $\cv{}$:
    $\critical1$ is given by the vertical strips in light color; $\critical2$ is given by
    the darker region inside $\critical1$.  We highlight a ``fundamental domain'' of
    $\critical2$; the reader can check that this is an accurate depiction of $\critical1$
    and $\critical2$ by simple inspection of the definition.}
  \label{f_criticals}
\end{figure}
The proof of the lemma will be given in appendix~\ref{s_proofLemmaLebesgue}; for clarity,
let us first introduce the following natural notion: given a basic curve $\Gamma$ and a
point $(\coo{})\in\Gamma$, for any $r>0$, we let the \emph{$\Gamma$-ball of radius $r$
  around $(\coo{})$} be the set of points $(\ct{}',\cv{}')\in\Gamma$ such that
$|\ct{}'-\ct{}|<r$; this induces the corresponding notion of \emph{$\Gamma$-neighborhood of a
  subset of\,\,$\Gamma$}.
\begin{lemma}\label{l_lebesgueCritical}
  The critical sets enjoy the following properties:
  \begin{itemize}
  \item[($a_1$)] for fixed $\Delta_1>0$, we can choose $\hat K_1$ so large that, for any standard
    curve $\Gamma$, the $\Gamma$-neighborhood of radius $r=\Delta_1\V_\Gamma^{-\beta}$ of
    $\Gamma\cap\critical1$ is contained in $\hcritical1$;
  \item[($a_2$)] for fixed $\Delta_2>0$, we can choose $\hat K_2$ so large that, for any
    standard curve $\Gamma$, the intersection of the $\Gamma$-neighborhood of radius
    $r=\Delta_2\V_\Gamma^{-2\beta}$ of $\Gamma\cap\critical2^*$ with $\hcritical1$ is
    contained in $\hcritical2$.
  \item[($b_1$)] for any standard pair $\ell$ we have
      $\prob_\ell(\hcritical{1})=\bigo{\V_\ell^{-\beta}}$;
  \item[($b_2$)] for any standard curve $\Gamma$, the number of connected components of
    $\Gamma\cap\hcritical{2}$ is bounded uniformly in $\V_\Gamma$.
  \item[($c_1$)] the Lebesgue measure of $\hcritical{1}$ is finite if $\gamma>3$;
  \item[($c_2$)] the Lebesgue measure of $\hcritical{2}$ is finite if $\gamma>2$;
  \end{itemize}
\end{lemma}

As we mentioned at the beginning of this subsection, on critical sets we lack good
geometrical and regularity estimates that can be achieved on the complementary set. In
particular, outside $\critical{1}$ standard pairs will be mapped to standard pairs; pieces
of standard pairs passing through the first critical set will possibly be mapped to
non-standard pairs. However, pieces of standard pairs that lie in $\hcritical{1}\setminus
\critical{2}^*$ are guaranteed to be standard after one more iteration. In the following
lemma we prove the previous statements and establish some expansion bounds which will be
crucial for proving equidistribution properties of $F$ along the horizontal direction.

\begin{definition}
  A \emph{standard partition} $\mathcal{J}$ of the circle $\sone$ is a partition$\mod 1$
  in a finite number of closed intervals $\mathcal{J}=\{J_\alpha\}$, $\alpha\in\aset$
  satisfying the following conditions:
  \[
  \intr J_\alpha\cap  \intr J_{\alpha'}=\emptyset\textrm{ if }\alpha\not = \alpha',\quad
  \sone=\bigcup_{\alpha\in\aset} J_\alpha,\quad \delta/4< |J_\alpha|< \delta/2.
  \]
  A basic pair $\ell$ is said to be $\mathcal{J}$\emph{-aligned} if  $I_\ell\in\mathcal{J}$.
\end{definition}

\begin{lemma}[Invariance]
  \label{l_invariance}
  Fix a standard pair $\ell=(\Gamma_\ell,\rho_\ell)$ and a standard partition $\mathcal{J}$; let $\V=\V_\ell$, $\T=\T(\V)$ and similarly for $\T'$. Then we can choose $\cv{*}$ large enough so that:
  \begin{enumerate}
  \item[(a)] the following estimates hold:
    \begin{subequations}
    \begin{align}
      \left|\left.\de{\ct{1}}{\ct{0}}\right|_{\Gamma_\ell}\right|&>\frac{1}{2}K_1{\T'}^{1/2}&\ {\rm if}\ (\coo{0})&\not\in \critical{1}\label{e_a1}\\
      \left|\left.\de{\ct{1}}{\ct{0}}\right|_{\Gamma_\ell}\right|&>\frac{1}{2}\bar K_2>1&{\rm if}\ (\coo{0})&\not\in \critical{2}\label{e_a2}\\
      \left|\left.\de{\ct{2}}{\ct{0}}\right|_{\Gamma_\ell}\right|&>\frac{1}{2}K_2 \T'&{\rm if}\ (\coo{0})&\in \hcritical{1}\setminus \critical{2}^*\label{e_a3}
    \end{align}
  \end{subequations}
\item[(b)] we can uniquely decompose $F\ell$ as follows:
  \begin{align}\label{e_invarianceDecomposition}
      F\ell &= \bigcup_{\alpha\in\mathcal{A}}\bigcup_{j}\ell_\alpha^j \cup \ell_+ \cup \ell_- \cup \bigcup_{j}\tilde\ell_j\cup Z
    \end{align}
    such that:
    \begin{itemize}
    \item each $\ell_\alpha^j$ is a $\mathcal{J}$-aligned standard pair and $I_{\ell_\alpha^j}=J_\alpha$
    \item $\ell_+$ and $\ell_-$ might be either empty or standard pairs such that $F\inv\ell_\pm\cap\critical{1}=\emptyset$.
    \item each $\tilde\ell_j$, which we call a \emph{stand-by pair}, is such that we have
      $F\tilde\ell_j=\bigcup_l\ell_{j,l}$ where $\ell_{j,l}$ are standard pairs;
    \item the number of stand-by pairs is bounded uniformly in $y$;
    \end{itemize}
    Moreover:
    \begin{equation}\label{e_invarianceInclusion}
      \ell\cap\critical{1}\subset F\inv\tilde\ell_j \subset \ell\cap\hcritical{1}\qquad
      \ell\cap\critical{2}\subset F\inv Z \subset \ell\cap\hcritical{2}
    \end{equation}
\end{enumerate}
\end{lemma}
\begin{proof}
  Recall that by definition:
  \begin{align*}
    \left.\de{\ct{1}}{\ct{0}}\right|_{\Gamma_\ell}(x)&=\LL_\ell(x)=\tslope{\ell}(x)\T'_\ell(x);
  \end{align*}
  moreover, if $(\coo{})\not\in\critical{1}$ we have that $|\tslope{1}(\coo{})|\geq K_1\T'(y)^{-1/2}$, and if $(\coo{})\not\in\critical{2}$ we have that $|\tslope{1}(\coo{})|\geq \bar K_2\T'(y)\inv$. Then, by \eqref{e_standardPairSlope}  we immediately obtain \eqref{e_a1} and \eqref{e_a2} provided that $\cv{*}$ is large enough.
  Additionally, we can conclude that outside $\corecritical{2}$ we have $\slope\ell(\ct{})\not = 0$, thus we can apply Lemma~\ref{l_basicCurveIteration} and using \eqref{e_a2}:
  \begin{equation}
    |\slope{\ell'}(\ct{1})-\slope{1}(\coo{1})|\leq 2(\bar K_2\T')^{-1};
    \label{e_firstIteration}
  \end{equation}
  hence, since $\ell$ is standard we obtain the following bound if $(\coo{0})\in\hcritical{1}\setminus\critical{2}^*$:
  \[
  |\tslope{\ell}(\ct{0})\tslope{\ell'}(\ct{1})|\geq 3/4 K_2\T'^{-1}
  \]
  which implies \eref{e_a3} since
  \begin{align*}
    \left.\de{\ct{2}}{\ct{0}}\right|_{\Gamma_\ell}&=\tslope{\ell}(\ct{0})\T'_\ell(\ct{0})\tslope{\ell'}(\ct{1})\T'_{\ell'}(\ct{1}).
  \end{align*}

  In order to prove part (b)
  first of all notice that if $(\coo{0})\not\in\critical{1}$ we can apply Lemma~\ref{l_basicCurveIteration} and part (a) obtaining:
  \begin{subequations}\label{e_singleIteration}
    \begin{align}
      |\slope{\ell'}(\ct{1})-\slope{1}(\coo{1})|&\leq 2K_1^{-1}{\T'}^{-3/2}\\
      |\dslope{\ell'}(\ct{1})-\dslope{1}(\coo{1})|&= \bigo{\T'^{-3/2}}\\
      |\rr{\ell'}(\ct{1})|&\leq 3 A\cdot 4K_1^{-2}+\bigo{\T'^{-1/2}}.
    \end{align}
  \end{subequations}
  Therefore, by taking $K_1$ sufficiently large and assuming $\cv{*}$ large enough, we can
  ensure that equations \eqref{e_definitionStandardPair} hold and that $\rho_{\ell'}$ is
  regular.

  On the other hand if $(\coo{0})\in\hcritical{1}\setminus\critical{2}^*$ we have, once more by Lemma~\ref{l_basicCurveIteration} and part (a):
    \begin{align*}
      |\slope{\ell'}(\ct{1})-\slope{1}(\coo{1})|&\leq 2\bar K_2^{-1}{\T'}^{-1}\\
      |\dslope{\ell'}(\ct{1})-\dslope{1}(\coo{1})|&= 3A\cdot 2\bar K_2^{-1} + \bigo{\T'^{-2}}\\
      |\rr{\ell'}(\ct{1})|&\leq 3 A\cdot \LL_\ell^{-2}\T'+\bigo{1},
    \end{align*}
  from which we obtain that $\tslope{\ell'}\not = 0$, and $|\LL_{\ell'}|>1/2K_1\T'^{1/2}$ so that we can apply Lemma~\ref{l_basicCurveIteration} to $\ell'$ and obtain:
  \begin{subequations}\label{e_doubleIteration}
    \begin{align}
      |\slope{\ell''}(\ct{2})-\slope{1}(\coo{2})|&\leq  2K_1^{-1}{\T'}^{-3/2}\\
      |\dslope{\ell''}(\ct{2})-\dslope{1}(\coo{2})|&= \bigo{\T'^{-3/2}}\\
      |\rr{\ell''}(\ct{2})|& \leq {3A\T'}{\LL_\ell^{-2}\LL_{\ell'}\inv} + 3A\cdot 2\bar K_2\inv\T'\LL_{\ell'}^{-2}  +\bigo{\T'^{-1/2}}\leq \label{e_doubleIterationR}\\
      & \leq 3A\cdot 4\bar K_2\inv (K_2\inv + 2 K_1^{-2})+\bigo{\T'^{-1/2}}
    \end{align}
  \end{subequations}
  which agree with equations \eqref{e_definitionStandardPair} and prove that
  $\rho_{\ell''}$ is regular provided we take large enough $K_1$, $K_2$ and $\cv{*}$.

  In order to conclude we need to carefully consider several possibilities: first assume
  that $\Gamma_\ell\cap\critical{1}=\emptyset$ and cut the image of $\Gamma_\ell$ in as
  many $\mathcal{J}$-aligned curves as possible; in doing so we might be left with two
  boundary curves, that we denote by $\ell_-^*$ and $\ell_+^*$. Consider for instance
  $\ell_-^*$: there are two possibilities; if $|I_{\ell_-^*}|>\delta/4$ we can simply let
  $\ell_-=\ell_-^*$; otherwise we let $\ell_-$ be the union of $\ell_-$ with the adjacent
  pair; since the latter is $\mathcal{J}$-aligned, we obtain that $\ell_-$ is a standard
  pair since $|I_{\ell_-}|<3/4\delta$; performing the same construction with $\ell_+^*$ we
  can conclude with:
  \[
  F\ell = \bigcup_{\alpha\in\mathcal{A}}\bigcup_{j}\ell_\alpha^j \cup \ell_+ \cup \ell_-,
  \]
  which concludes the proof of item (b) assuming that $\Gamma_{\ell}\cap\critical{1} =
  \emptyset$.

  Assume now that $\Gamma_\ell\cap\critical{1}\not = \emptyset$; then by our choice of
  $\delta$ we know that $\Gamma_\ell\setminus\critical{1}$ has at most two connected
  components, that we denote by $\Gamma_1$ and $\Gamma_2$; in turn let
  $\Gamma_{*}=\Gamma_{\ell}\cap\critical{1}$. We will consider $\Gamma_1$ and $\Gamma_2$
  separately; to fix ideas let us work with $\Gamma_1$. Assume first that $|I_1|>4\pi
  K_1\inv\T'^{-1/2}$; then, as before, we can cut the image of $\Gamma_{1}$ in as many
  $\mathcal{J}$-aligned curves as possible plus two boundary curves. One of them will not
  contain the image of $\partial\critical{1}$ whereas the other one will necessarily
  do. As before, we let the former to be $\ell_-$, joining it with the adjacent one if it
  turns out to be too short; the preimage of the latter will be instead joined to
  $\Gamma_*$; if, on the other hand $|I_1|\leq 4\pi K_1\inv\T'^{-1/2}$, then we join the
  whole $\Gamma_1$ to $\Gamma_{*}$. We do the same with the other connected component.
  Thus, as before we have
  \[
  F\ell = \bigcup_{\alpha\in\mathcal{A}}\bigcup_{j}\ell_\alpha^j \cup \ell_+ \cup \ell_-
  \cup F\ell_{*}
  \]
  and we are left with $\Gamma_*$ such that $|I_*|>3/2\delta K_1\inv\T'^{-1/2}$.  By
  taking $\hat K_1$ sufficiently large we can ensure that
  $\ell_*\subset\hcritical{1}$. Consider now $\Gamma_*\setminus\critical{2}^*$; by
  Lemma~\ref{l_lebesgueCritical}, this set has a uniformly bounded number of connected
  components; consider each connected component. If it is longer than $2
  K_2\inv{\T'}\inv$, we let its image be one of the $\tilde\ell_j$; by our previous
  arguments the image of $\tilde\ell_j$ can indeed be decomposed in standard pairs. We
  thus choose $\Delta$ so large that all short components will belong to $\hcritical2$,
  which allows us to conclude.
\end{proof}
We now introduce the notion of \emph{critical time}; for a fixed standard partition $\mathcal{J}$, for any standard pair $\ell$ the critical time of a point $p\in\Gamma_\ell$ is the largest number $\bar{n}$ such that, by iterating the decomposition in lemma \ref{l_invariance}, $F^{n}p$ belongs to a non-invalid curve for all $n\leq \bar n$.

\begin{definition}\label{d_criticalTime}
  Fix a standard partition $\mathcal{J}$ and let $\ell$ be a standard pair. We define the \emph{critical time} as a function $\tau_\ell:\Gamma_\ell\to\naturals\cup\{\infty\}$ obtained by means of the following recursive definition:
  let $p\in\Gamma_\ell$, then by item (b) of lemma \ref{l_invariance} we have three possibilities:
  \begin{itemize}
  \item $Fp$ belongs to a standard pair $\ell'$: we then define $\tau_\ell(p)=\tau_{\ell'}(Fp)+1$;
  \item $Fp$ belongs to a stand-by pair, hence $F^2p$ belongs to a standard pair $\ell''$: we define $\tau_\ell(p)=\tau_{\ell''}(F^2p)+2$;
  \item otherwise we define $\tau_\ell(p)=0$.
  \end{itemize}
\end{definition}
The following proposition is the crucial technical result of our work.
\begin{prp}\label{l_criticalTimeEstimate}
  If $\gamma>2$, for any standard pair $\ell$, we have $\prob_\ell(\tau_\ell<\infty)=1$.
\end{prp}
The proof will be given in Section~\ref{s_reduction}. We will now show how it implies our
Main Theorem; the argument is a trivial adaptation of the analogous one found in
\cite{Dima}; we give it here for completeness.
\begin{proof}[Proof of the Main Theorem]
  First of all notice that Lebesgue measure can be disintegrated in reference pairs,
  i.e. for any $E$ Borel measurable set:
  \[
  \leb{E}=\int\prob_{\ell_\alpha}(E)\deh\lambda_\alpha
  \]
  where $\deh\lambda_\alpha$ is some factor measure on reference pairs. Furthermore, notice that by definition of $\tau_\ell$ and by \eqref{e_invarianceInclusion}, Lemma~\ref{l_criticalTimeEstimate} immediately implies that:
  \[
  \prob_\ell(\{(\coo{0})\in\Gamma_\ell\st(\coo{n})\not\in\hcritical{2}\ \fa n\in\naturals\})=0.
  \]
 Hence we obtain:
  \begin{equation}\label{e_zeroLebMeasure}
    \leb{\{(\coo{0})\st(\coo{n})\not\in\hcritical{2}\ \fa n\in\naturals\}}=0.
  \end{equation}
  Define now $\hat F:\hcritical{2}\to\hcritical{2}$ as the first return map of $F$ on
  $\hcritical{2}$; $\hat F$ is well defined almost everywhere by \eqref{e_zeroLebMeasure};
  moreover, lemma \ref{l_lebesgueCritical} implies that $\leb{\hcritical{2}}<\infty$,
  consequently we can apply Poincar\'e recurrence theorem and conclude that almost every
  point in $\hcritical{2}$ is recurrent, which shows that
  $\leb{\escaping\cap\hcritical{2}}=0$.  This implies our Main Theorem since, using once
  more \eqref{e_zeroLebMeasure}, we know that the orbit of almost every point in
  $\phspace_*$ intersects $\hcritical{2}$.
\end{proof}

%% file: Criticals.pdf_tex
\begingroup%
  \makeatletter%
  \providecommand\color[2][]{%
    \errmessage{(Inkscape) Color is used for the text in Inkscape, but the package 'color.sty' is not loaded}%
    \renewcommand\color[2][]{}%
  }%
  \providecommand\transparent[1]{%
    \errmessage{(Inkscape) Transparency is used (non-zero) for the text in Inkscape, but the package 'transparent.sty' is not loaded}%
    \renewcommand\transparent[1]{}%
  }%
  \providecommand\rotatebox[2]{#2}%
  \ifx\svgwidth\undefined%
    \setlength{\unitlength}{256.00505371bp}%
    \ifx\svgscale\undefined%
      \relax%
    \else%
      \setlength{\unitlength}{\unitlength * \real{\svgscale}}%
    \fi%
  \else%
    \setlength{\unitlength}{\svgwidth}%
  \fi%
  \global\let\svgwidth\undefined%
  \global\let\svgscale\undefined%
  \makeatother%
  \begin{picture}(1,0.57012362)%
    \put(0,0){\includegraphics[width=\unitlength]{Criticals.pdf}}%
    \put(0.65892699,0.13104164){\color[rgb]{0,0,0}\makebox(0,0)[lb]{\smash{$\ct{1}=0\mod2\pi$}}}%
    \put(0.65892699,0.22478982){\color[rgb]{0,0,0}\makebox(0,0)[lb]{\smash{$\ct{1}=\pi\mod2\pi$}}}%
    \put(0.65892699,0.32478784){\color[rgb]{0,0,0}\makebox(0,0)[lb]{\smash{$\ct{1}=0\mod2\pi$}}}%
    \put(0.40941977,0.00509632){\color[rgb]{0,0,0}\makebox(0,0)[b]{\smash{$\ct{0}$}}}%
    \put(0.14315513,0.35201373){\color[rgb]{0,0,0}\makebox(0,0)[rb]{\smash{$\cv{0}$}}}%
  \end{picture}%
\endgroup%

%% file: equidistribution.tex
\input{definitions.p}

\newcommand{\ellref}[2]{\ell_{#1}^{#2}}
\newcommand{\onehalf}{\frac{1}{2}}
\newcommand{\lS}{\V}
\newcommand{\uS}{{\V^*}}
\newcommand{\iTheta}{\xi}
\newcommand{\thi}{\xi}
\newcommand{\baseSlope}[1]{\tslope{S}(#1)}
\newcommand{\Cneps}[1]{C_{#1,\eps}}
\newcommand{\kset}[1]{\mathcal{J}_{#1}}
\newcommand{\Kset}[1]{\mathbf{J}_{#1}}
\newcommand{\prt}[2]{\langle #1 \rangle_{#2}}

\section{Equidistribution}\label{s_equidistribution}
In this section we set up an induction scheme to prove equidistribution estimates on
standard pairs for a specific class of observables. The observables we consider are
sufficiently smooth function of the fast variable $\ct{}$ which are constant on the
$\cv{}$ direction.

In the sequel, we will often need to approximate integrals of such observables with
Riemann sums (or viceversa) over partitions which are highly non-uniform. Most element of
the partition will have small size compared to a much smaller portion of them which have
sizes that are order of magnitudes larger. The na\"\i ve bound on the Riemann sum, which
is optimal for uniform partitions, gives estimates which are not sufficient for our
purposes. The following lemma\footnote{The original proof of this lemma was substantially
  more involved; I am, again, indebted to D. Dolgopyat for providing me with the much more
  elegant argument which is used here.} will be systematically
used to obtain crucial estimates.
\begin{lemma} \label{l_E0} Let $(\Omega,\mu)$ be a finite measure space and
  $\RiemFunc:\Omega\to[0,1]$ a measurable function. Assume there exist real numbers
  $0<\lambda<1$, $C>0$ and $0<\alpha\leq 1$ such that for any $1\leq z \leq \lambda\inv$:
  \[
  \mu\{f>z\lambda\}\leq \mu(\Omega)Cz^{-\alpha}.
  \]
Then:
\[
\mu(f)\leq \mu(\Omega)(C+1)
\begin{cases}
  \frac1{1-\alpha}\lambda^{\alpha}&\text{if}\ \alpha<1\\
  \lambda|\log\lambda|&\text{if}\ \alpha=1.
\end{cases}
\]
\end{lemma}
\begin{proof}
  Let $\hat f =\max(\lambda,f)$; then
  \[
  \frac{\mu(f\lambda\inv)}{\mu(\Omega)}\leq \frac{\mu(\smash{\hat f}\lambda\inv)}{\mu(\Omega)}\leq
  1+ C\int_1^{\lambda\inv}z^{-\alpha}\leq 1+C
  \begin{cases}
    \frac{\lambda^{\alpha-1}}{1-\alpha}&\text{if}\ \alpha<1\\
    -\log\lambda&\text{if}\ \alpha=1.
  \end{cases}
  \]
\end{proof}
Given a standard pair $\ell$, recall that we denote by $\mathcal{L}_\ell$ the expansion
rate \smash{$\left|\de{\ct{1}}{\ct{0}}\right|$} along $\Gamma_\ell$ and define:
\[
\hat{\mathcal{L}}_\ell\defeq\inf_{\Gamma\setminus\critical{1}}\LL_\ell;
\]
moreover, define $\beta$ so that $\gamma=2\beta+1$.  We will often use the conventional
notation $\const$ to indicate some positive real number which does not depend on $y$ or
other indices; the actual value of $\const$ can change from expression to expression.
\begin{lemma}[Base equidistribution step]\label{l_baseEquiStep}
  There exists a constant $C$ such that, given a standard pair $\ell$,
  $\averaged\in\continuous{}{\sone}$ and $B\in\continuous{1}{\Gamma_\ell}$, we have:
  \begin{align}\label{e_aux1}
    |\expectation_\ell(B\cdot\averaged\circ F) -
    \expectation_\ell(B)\langle\averaged\rangle|\leq\|\averaged\|\Big(&\|B\|\big(\prob_\ell(\critical{1})+C\hat\LL^{-1}_\ell\big)+\notag\\
    &+\|\dot B\| C\V_\ell^{-2\beta}\log\V_\ell\Big)
  \end{align}
  where $\averaged\circ F$ is a shorthand notation for $\averaged(\pi
  F(\ct{},\psi_\ell(\ct{})))$ and
  $\langle\averaged\rangle=\int_0^{2\pi}\averaged(\theta)\deh\theta$.
\end{lemma}
Notice that by linearity of expectation we can always assume that $\averaged$ has zero
average.  The lemma ensures that, in one step, the dynamics acts on standard pairs by
making then approach Lebesgue measure for observables which are independent of $y$.
The lemma will be proved by means of the following slightly more general version;
\newcommand{\nqt}{\varphi} 

\begin{lemma}\label{p_auxBaseStep}
  Let $I$ be a standard interval and $\rho$ a regular probability density on $I$
  associated to the probability measure $\prob$; let $\nqt:I\to\reals$ be a smooth
  function with at most one non-degenerate critical point and normalized so that
  $\|\dot\varphi\|\leq 1$; for $\LL\gg 1$ sufficiently large, let $\Theta:I\to\torus$ given by
  $\Theta=\LL\nqt\mod 2\pi$.  Let $D=\{|\dot\Theta|<\LL^{1/2}\}$ and
  $\hat\LL=\inf_{I\setminus D}|\dot\Theta(x)|\leq\LL$. Then there exist $C>0$ which
  does not depend on $\LL$, such that for any $B\in\continuous{1}{I}$,
  $\averaged\in\continuous{}{\torus^1}$ a zero average function:
    \begin{align}\label{e_aux2}
    \left|\int_I B(x) \averaged(\Theta(x))\rho(x)\deh x\right|\leq\|\averaged\|\Big(&\|B\|\big(\prob(D)+C\hat\LL^{-1}\big)+\notag\\
    +&\|\dot B\|_{I\setminus D} C\LL^{-1}\log\LL\Big).
  \end{align}
  \end{lemma}
  \begin{proof}
    Cut $I\setminus D$ at the points $\Theta = 0\mod2\pi$ and denote by $\{J_k\}$ the set of intervals in $I\setminus D$ bounded by two consecutive cutting points. We obtain the following bound for the leftover pieces:
  \begin{equation}
    \prob(D)\leq \prob(I \setminus\bigcup_k J_k)\leq \prob(D)+K\hat\LL^{-1}.
    \label{e_leftover}
  \end{equation}
  The left inequality is obvious; for the right one notice that the number of leftover
  (connected) pieces is bounded by twice the number of connected components of $I
  \setminus D$, that can be at most 2 by definition of $D$. The measure of each of such
  pieces can in turn be bounded using Lemma~\ref{l_comparison} to obtain \eref{e_leftover}
  with $K=4\mu_2$.

  Let
  \[
  E_k = \int_{J_k} B(x) \averaged(\Theta(x))\rho(x)\deh x;
  \]
  for each $k$, let $\thi_k(\theta)$ be the inverse function of $\Theta$ on $J_k$; moreover define the pushforward $\rho'_k(\theta)=\rho(\thi_k(\theta))/|\dot\Theta(\thi_k(\theta))|$ and the auxiliary function given by $H_k(\theta)=B(x(\theta))\rho'_k(\theta)$. Then:
  \[
  E_k=\int_0^{2\pi}H_k(\theta)\averaged(\theta)\deh\theta.
  \]
  Write $H_k(\theta)=\bar H_k+\tilde H_k(\theta)$, where $\bar H_k$ is the average of $H_k$. Since $\averaged$ has zero average we obtain:
  \[
  |E_k|=\left|\int_0^{2\pi}\tilde H_k(\theta)\averaged(\theta)\deh\theta\right|\leq 2\pi\|\averaged\| \|\tilde H_k\|.
  \]
  Fix $k$ and let $\theta_0$ be such that $\tilde H_k(\theta_0)=0$, then:
  \[
  |\tilde H_k(\theta)|\leq\int_{\theta_0}^{\theta}\left|\de{\tilde H_k}{\theta}\right|\deh s \leq\int_{\theta_0}^{\theta}\rho'_k\left( |B|\left| \de{\log\rho'_k}{\theta} \right| + \left| \frac{\dot B}{\dot\Theta}\right| \right)\deh{s};
\]
thus, if we let $c_k=\prob(J_k)$ we obtain:
\begin{equation}\label{e_equiIntermediate}
  |E_k|\leq 2\pi\|\averaged\|c_k\left( \|B\|_{I\setminus D}\|r'\|_{J_k}+\|\dot B\|_{I\setminus D}\|\dot\Theta^{-1}\|_{J_k} \right).
\end{equation}
where $r'=r\dot\Theta^{-1} + \ddot\Theta\dot\Theta^{-2}$; let $\hat\LL_k=\inf_{x\in
  J_k}|\dot\Theta(x)|\geq\hat\LL$, so that:
\begin{equation}\label{e_estimateR}
\|r'\|_{J_k} \leq \|r\|\hat\LL\inv_k+\|\ddot\Theta\|\hat\LL^{-2}_k\leq \const(\hat\LL\inv_k+\LL\hat\LL^{-2}_k).
\end{equation}
We thus need to obtain a bound for the two sums $\sum_kc_k\hat\LL\inv_k$
and$\sum_kc_k\hat\LL^{-2}_k$ and to this purpose we are going to use Lemma~\ref{l_E0}; on
the one hand, for any $c$, if $\LL_k>c\LL$, then the two sums are bounded above by
$\const\LL\inv$ and $\const\LL^{-2}$ respectively. On the other hand, since critical
points of $\nqt$ are non degenerate, there exist $c_\nqt,C_\nqt>0$ such that, if
$|\dot\nqt|<2c_\nqt$, then $\dot\nqt$ is monotone and for any $z\geq 1$:
\[
\prob(|\dot\nqt|<c_\nqt z\inv)\leq C_\nqt z\inv
\]
Since the diameter of each $J_k$ is bounded above by $\const\hat\LL\inv$, and using
monotonicity of $\dot\nqt$ and boundedness of $\ddot\nqt$, the above estimate implies:
\begin{equation}\label{e_niceBoundL}
\prob(\hat\LL\hat\LL\inv_k>\hat\LL\LL\inv c_\nqt\inv z)\leq C_\nqt z\inv + \const\hat\LL\inv.
\end{equation}
Let now $\Omega$ be the finite measure space whose elements are the intervals $J_k$ with
measure $c_k$; let $f:J_k\mapsto\hat\LL\hat\LL\inv_k$; then we can apply Lemma~\ref{l_E0}
to $f$, since by construction $f(J_k)\in[0,1]$ and \eqref{e_niceBoundL} implies that
\[
\prob(f>\lambda z)\leq C_\nqt z\inv + \const\hat\LL\inv \leq \const z^{-\alpha}
\ \text{with}\ \alpha=1
\]
with $\lambda=\hat\LL\LL\inv c_\nqt\inv$, where the second inequality holds because, by
construction, $z\inv\geq\lambda = \const \LL^{-1/2}$.  We therefore obtain:
\[
\sum_k c_k\hat\LL\inv_k\leq \const \LL\inv\log\LL
\]
Similarly, we apply Lemma~\ref{l_E0} to $f^2$, using $\lambda=\hat\LL^2\LL^{-2}
c_\nqt^{-2}$ and $\alpha=1/2$ obtaining:
\[
\sum_k c_k\hat\LL^{-2}_k\leq \const \LL\inv\hat\LL\inv.
\]
Plugging the above estimates in \eqref{e_equiIntermediate} and using \eqref{e_estimateR},
we obtain \eqref{e_aux2} and finally conclude the proof.
\end{proof}
\begin{proof}[Proof of Lemma~\ref{l_baseEquiStep}]
  Let $\Theta_\ell(\ct{})=\ct{}+\T_\ell(\ct{})$; then we have by definition:
  \[
  \expectation_\ell(B\cdot\averaged\circ F)=\int_{I_\ell}B(\ct{},\psi_\ell(\ct{}))\cdot\averaged(\Theta_\ell(\ct{}))\rho_\ell(\ct{})\deh\ct{}.
  \]
  Let $\LL=\const\V_\ell^{2\beta}$ and define $\nqt=\Theta_\ell/\LL$; the reader will not
  find difficult to prove that $\nqt$ satisfies the hypotheses of
  Lemma~\ref{p_auxBaseStep}, which implies our Lemma.
\end{proof}
\begin{cor}\label{c_lazyInduction}
  There exists a constant $C$ such that, for any $\ell$, $\averaged$ and $B$ as in the
  statement of lemma~\ref{l_baseEquiStep} and $n\geq0$ we have:
  \begin{align}\label{e_lazyInduction}
     |\expectation_\ell(B\cdot\averaged\circ F^{n+1}1_{\tau\geq n})&-\expectation_\ell(B\cdot 1_{\tau\geq n})\langle\averaged\rangle|\leq\notag\\ &\leq\|\averaged\|\left(\|B\|C\V_\ell^{-\beta}+\left\|\de{B}{\ct{n}}\right\|_* C\V_\ell^{-2\beta}\log\V_\ell\right).
   \end{align}
   where the norm $\|\cdot\|_*$ is the sup restricted on those points which are mapped to
   a standard pair after $n$ iterates.
\end{cor}
\begin{proof}
  If $n=0$, the corollary trivially follows from lemma~\ref{l_baseEquiStep}; we
  henceforth assume that $n>0$ and, as before, that $\langle\averaged\rangle=0$.
  Iterate lemma~\ref{l_invariance} for $n$ times and obtain:
  \[
  F^{n}\ell=\bigcup_j\ell'_j\cup \bigcup_k\tilde\ell_k\cup\{\tau< n\};
  \]
  moreover we know that $F^{-1}\tilde\ell_k\subset \ell^*\cap\hcritical{1}$ where $\ell^*$
  is standard.  Thus, by Lemma~\ref{l_lebesgueCritical} we have:
  \[
  \prob_\ell(F^{-n}\bigcup_k\tilde\ell_k)\leq \const\cdot\V_\ell^{-\beta}.
  \]
  We can hence obtain \eref{e_lazyInduction} by applying Lemma~\ref{l_baseEquiStep} to
  each standard pair $\ell'_j$.
\end{proof}
\begin{rmk}\label{r_clearEqui}
  If we choose $B$ equal to the constant function $1$ we obtain, by
  Lemma~\ref{l_baseEquiStep} and Corollary~\ref{c_lazyInduction} that there exists a $C>0$
  such that:
  \begin{subequations}
    \begin{equation}
      |\expectation_\ell(\averaged\circ
      F)-\langle\averaged\rangle|\leq\|\averaged\|\big(\prob_\ell(\critical{1})+C\mathcal{L}^{-1}_\ell\log\V_\ell\big).\label{e_baseEquiStep}
    \end{equation}
    \begin{equation}
      |\expectation_\ell(\averaged\circ F^n1_{\tau\geq n-1})-\langle\averaged\rangle\prob_\ell(\tau\geq n-1)|\leq\|\averaged\|C\V_\ell^{-\beta}.
    \end{equation}
  \end{subequations}
\end{rmk}
We  need to perform a substantially more accurate analysis in order to improve
estimates \eqref{e_lazyInduction}. This is the principal technical difficulty of our
work. Said analysis, which is the main result of this section, is summarized in the following
\begin{lemma}[Equidistribution lemma]\label{l_equidistribution}
  For all $\beta>1/2$ there exists $\nu(\beta)\in\naturals$ such that for any $n\geq\nu$, any sufficiently smooth function $\averaged$ and any standard pair $\ell$ with $\V_\ell$ large enough, we have:
  \begin{equation}\label{e_usefulEqui}
    |\expectation_\ell(\averaged\circ F^n1_{\tau\geq n-1})-\langle\averaged\rangle\prob_\ell(\tau\geq n-1)|\leq\|\averaged\|'C_n\,o(\V_\ell^{-1}).
  \end{equation}
    where $\|\averaged\|'$ is given by $\sum_{l}l^2|\hat\averaged_l|$ with
    $\hat\averaged_l$ the $l$-th Fourier coefficient of $\averaged$ and $C_n$ is uniform
    in $\V_{\ell}$.
  \end{lemma}
  Notice that if $\beta>1$, then Lemma~\ref{l_equidistribution} immediately follows by
  Remark~\ref{r_clearEqui} by taking $\nu=1$; this is essentially the work of \cite{Dima}.
  In order to prove Lemma~\ref{l_equidistribution} for smaller values of $\beta$ we will
  show the following result: if $\beta>1/2$, then for $n\geq 2$, there exists a constant
  $C_{n}$ such that for any sufficiently smooth function $\averaged$ with
  $\langle\averaged\rangle=0$ and standard pair $\ell$ as in the statement of
  Lemma~\ref{l_equidistribution} we have:
\begin{equation} \label{e_equi}
  |\expectation_\ell(\averaged\circ F^n\cdot 1_{\tau \geq n-1})|\leq
  \|\averaged\|' C_{n}(\V_\ell^{-\beta-(n-1)(\beta-\onehalf)}+o(\V_\ell^{-1})).
  \end{equation}
  Lemma~\ref{l_equidistribution} then follows from \eqref{e_equi} by choosing
  \begin{equation*}
    \nu>\frac{1}{2}\left(\beta-\frac{1}{2}\right)^{-1}.
\end{equation*}
In the remaining part of this section we will always assume $1/2<\beta\leq 1$;
additionally, once $\beta$ is fixed, we assume $n$ to be fixed as well: in fact, our
construction depends on $n$ in that we are required to take $\ell$ with larger $\V_\ell$
as $n$ grows and the constant $C_{n}$ appearing in \eqref{e_equi} tends to infinity as
$n\to\infty$.  This will not be an issue since we will invoke
Lemma~\ref{l_equidistribution} with $n$ bounded as a function of $\beta$.

Our proof of Lemma~\ref{l_equidistribution} is based on two main ingredients: the
first one is an estimate of the contribution of the pieces of standard pairs which lie in
$\hcritical1$; the second one is a cancellation estimate for higher iterates of $F$
outside $\critical1$. The former is in fact stated in lemma \ref{l_nonStandard}; the
latter requires much finer estimates and will be described in the remaining part of this
section.

We now introduce some convenient definitions: a basic pair $\ell$ is said to be a
\emph{clean pair} if $\Gamma_\ell\setminus\critical{1}$ is connected. Given $\ell$ and $n$
we define a compact region of the phase space which contains $F^k\ell$ for
$k\in\{0,\cdots,n\}$. Let $\lS=\V_\ell-2A(n+1)$ and $\uS=\V_\ell+2A(n+2)$ and introduce
the notation $\baseSlope{\ct{}}=\tslope{1}(\ct{},\lS)$; we will always assume $\lS$ to be
large enough so that $\lS<\uS<2\lS$. Let $S=[\lS,\uS]$: we say that a standard pair $\ell$
is $(S,k)-$compatible if $F^l\Gamma_\ell\subset\torus\times S$ for $0\leq l\leq k$.

We now introduce a standard partition $\mathcal{I}=\{I_\alpha\}_{\alpha\in\aset}$ which
satisfies some useful
properties:
\begin{itemize}
\item each $I_\alpha$ is such that $|\baseSlope{\ct{}}|$ admits a unique minimum which we denote by $\bar x_\alpha$;
\item if $\ct{}\in\intr I_\alpha$, then $\dot\phi(x)\not =0$ and $\baseSlope{\ct{}}\not =0$;
\end{itemize}
In particular we have that any $\mathcal{I}$-adapted pair is a clean pair. Define the
strips $S_\alpha=I_\alpha\times S$; on each strip we define adapted coordinates
$\kappa_\alpha$ in such a way that $\kappa_\alpha(0,\eta)=(\bar{\ct{}}_\alpha,\eta)$. Let
$\ellref{\alpha}{\eta}=(\Gamma_{\alpha}^{\eta},\bar\rho_\alpha=1/|I_\alpha|)$ be the
reference pair on the curve given by the image under $\kappa_\alpha$ of the horizontal
line at height $\eta$, that is, the reference pair with base $I_{\alpha}$ passing through
the point $(\bar x_{\alpha},\eta)$. For $\alpha\in\aset$ let $\hat\LL_\alpha=\inf_{\eta\in S}\hat\LL_{\ell^{\eta}_\alpha}$ and for $1\leq k \leq n$ define the following function:
\begin{equation}\label{e_definitionPsi}
  \Psi_{\alpha,k}(\eta)=\expectation_{\ellref{\alpha}{\eta}}(\averaged\circ F^{k}\cdot 1_{\tau\geq k-1}).
\end{equation}
We now sketch the proof of \eqref{e_equi}: according to Lemma~\ref{l_invariance}, a
large portion of the image of a standard pair is given by a union of standard pairs; we
need to prove that the weighted sum of the expectations of $\averaged\circ F^{n-1}$ over
this union is of smaller order with respect to each term of the sum. In order to do so we
need first to prove that we can approximate the expectation on a given standard pair with
the expectation on an appropriate reference pair; hence we reduce to compute the weighted
sum of the expectations on a number of reference pairs, that is, a weighted sum of a
number of functions $\Psi$ defined in \eqref{e_definitionPsi}. We will prove that $\Psi$
are sufficiently regular and periodic in the variable $\T(\eta)$ up to a negligible error
$\bigo{\V\inv}$. This, and a fine control of the geometry of images of standard pairs
allows us to prove an estimate for the cancellation at each step, which will finally lead
to~\eqref{e_equi}.

We now state a number of lemmata which will be used to prove \eqref{e_equi}; their proofs
are quite technical and, as such, are postponed to the next subsection for easiness of
exposition.  We start with four lemmata related to the first iterate.
\begin{lemma}[Comparison I]\label{l_reductionBase}
  There exists $C>0$ such that for any standard pair $\ell$ and any reference pair
  $\bar\ell$ with $I_\ell=I_{\bar\ell}$ and
  $\Gamma_{\ell}\cap\Gamma_{\bar\ell}\not=\emptyset$ we have:
  \[
  |\expectation_{\vphantom{\bar\ell}\ell}(\averaged\circ
  F)-\expectation_{\bar\ell}(\averaged\circ F)|\leq
  C\|\averaged\|_1(\|r_\ell\|\V_\ell^{-\beta}+\T'(\V_\ell)\|\Delta\slope{\ell}\|),
  \]
  where $\|\cdot\|_1$ is the usual $\mathscr{C}^1$-norm.
\end{lemma}
\begin{lemma}[Periodicity]\label{l_periodicityBase}
  Assume that $\eta_0,\eta_1\in S$ with $\T(\eta_0) = \T(\eta_1)\mod 1$; then:
  \begin{equation}\label{e_periodicityBase}
    |\Psi_{\alpha,1}(\eta_1)-\Psi_{\alpha,1}(\eta_0) |= \|\averaged\|_1 o(\lS^{-1}).
  \end{equation}
\end{lemma}
\begin{lemma}[Differentiability]\label{l_lipshitzBase}
  Assume that $\eta\in S$; then there exists $C>0$ satisfying:
  \begin{equation}\label{e_lipschitzBase}
   \left|\de{\Psi_{\alpha,1}(\eta)}{\eta}\right|\leq C\|\averaged\|_1\,\T'(\eta)\hat\LL_{\alpha}^{-1}.
  \end{equation}
\end{lemma}

\begin{lemma}[Fourier components]\label{l_fourierEstimate}
  There exists $C>0$ and a sequence $\{\hat\Psi_{\alpha,1}^{(l)}\}_{l\in\integers}$ such
  that, for all $\eta\in S$:
\begin{equation}\label{e_globalFourierEstimate}
    \Psi_{\alpha,1}(\eta)=\sum_{l\in\integers}\hat\Psi_{\alpha,1}^{(k)}e^{-2\pi
      i\,l\T(\eta)} + \|\averaged'\|_1o(\lS^{-1}).
  \end{equation}
  where $\hat\Psi_{\alpha,1}^{(0)}=0$ and if $k\not = 0$:
    \begin{equation}\label{e_fourierEstimate}
    |\hat\Psi_{\alpha,1}^{(k)}|\leq C|\hat{\averaged}_k|\lS^{-\beta}
  \end{equation}
  \end{lemma}
We proceed with three analogous lemmata related to higher iterates.
\begin{lemma}\label{l_lipshitz}
 For all $2\leq k \leq n$, there exists a constant $C_{k}$ such that for all $\eta_1,\eta_2\in S$ with $|\T(\eta_1)-\T(\eta_2)|<1$ we have:
  \begin{equation}\label{e_lipshitz}
    |\Psi_{\alpha,k}(\eta_1)-\Psi_{\alpha,k}(\eta_2)|\leq\|\averaged\|_1\, C_{k} o(\lS^{-1}).
  \end{equation}
\end{lemma}
\begin{lemma}[Comparison II]\label{l_reduction}
   For all $2\leq k \leq n$ there exists $C_{k}$ such that for any $\ell$ clean standard pair $(S,n-k)$-compatible, there exists a reference pair $\bar\ell$ such that $I_\ell= I_{\bar\ell}$ satisfying
  \begin{align}
    |\expectation_\ell(\averaged\circ F^k\cdot 1_{\tau\geq k-1})&-\expectation_{\bar\ell}(\averaged\circ F^k\cdot 1_{\tau\geq k-1})|\leq\notag\\
    &\leq\|r_\ell\|\expectation_{\bar\ell}(\averaged\circ F^k\cdot 1_{\tau\geq k-1})+\|\averaged\|C_{k} o(\lS^{-1}).\label{e_reductionEstimate}
  \end{align}
\end{lemma}
\newcommand{\asets}{\aset^*}
\newcommand{\as}{{\alpha^*}}
\newcommand{\we}{\omega}
\begin{lemma}[Periodicity II]\label{l_periodicity}
  There exist a subset $\asets\subset\aset$, constants $\we_\as$ where $\as\in \asets$,
  constants $C_k$ where $2 \leq k \leq n$ and sequences of coefficients
  $\hat\Psi_{\alpha,n}^{(\as,l)}$, where $l\in\integers$ and $\as\in\asets$ such that for
  all $\eta\in S$ we have:
  \begin{equation}\label{e_periodicity}
    \Psi_{\alpha,k}(\eta)=\sum_{l,\as}\hat\Psi_{\alpha,k}^{(\as,l)}e^{2\pi i
      l\we_\as\T(\eta)} + \|\averaged\|'o(\V^{-1}).
  \end{equation}
  where there exists $C'_k$ such that for all $l\in\integers$:
    \begin{subequations}
      \begin{align}
        |\hat\Psi_{\alpha,2}^{(\as,l)}|&\leq C'_2\left(\V^{\onehalf-\beta}+l\cdot\V^{1-2\beta}\log\V\right)\max_{\alpha'\in\aset}|\hat\Psi_{\alpha',1}^{(l)}|;\label{e_mainDecaya}\\
        |\hat\Psi_{\alpha,k}^{(\as,l)}|&\leq C_{k}'\left(\V^{\onehalf-\beta}+l\cdot\lS^{-1}\right)\max_{\alpha'\in \aset}|\hat\Psi_{\alpha',k-1}^{(\as,l)}|\quad \fa 2<k\leq n;\label{e_mainDecayb}
      \end{align}
    \end{subequations}
    this implies that:
    \begin{equation}\label{e_explicitDecay}
      |\hat\Psi_{\alpha,k}^{(\as,l)}|\leq C_k|\hat\averaged_l|l\cdot(\lS^{-\beta-(k-1)(\beta-\onehalf)}+o(\lS^{-1})).
    \end{equation}
    Moreover $\we_\as$ are of order $\lS\inv$, that is there exist $C$ such that
    \begin{equation}\label{e_upperlowerWe}
    C\inv\lS\inv < \we_\as < C\lS\inv.
    \end{equation}
  \end{lemma}
  We now show how, given the above lemma, we can obtain \eqref{e_equi}; in fact by
  Lemma~\ref{l_reduction} we have that:
  \[
  |\expectation_\ell(\averaged\circ F^n\cdot 1_{\tau\geq n-1})|\leq 2|\expectation_{\bar\ell}(\averaged\circ F^n\cdot 1_{\tau\geq n-1})| +\|\averaged\|C_{n}o(\lS^{-1})
  \]
  then using \eqref{e_periodicity} and \eqref{e_explicitDecay} we conclude that
  \[
  |\expectation_{\bar\ell}(\averaged\circ F^n\cdot 1_{\tau\geq n-1})|\leq \|\averaged\|'(C_n\lS^{-\beta-(n-1)(\beta-\onehalf)}+o(\lS^{-1}))
  \]
  from which \eqref{e_equi} immediately follows.
  \subsection{Proofs of Lemmata~\ref{l_reductionBase}-\ref{l_periodicity}}
  First of all notice that applying Lemma~\ref{l_invariance} to $\ell$ and the standard
  partition $\mathcal{I}$, we can decompose the image of $\ell$ as:
  \[
  F\ell = \bigcup_{\alpha\in\aset}\bigcup_{j\in\mathbf{J_\alpha}}\ell_\alpha^j \cup \ell_+ \cup \ell_- \cup \bigcup_{j}\tilde\ell_j\cup\{\tau=0\}
  \]
  We begin with the following proposition, which allows to control the contribution to
  $\expectation_\ell(\averaged\circ F^n\cdot 1_{\tau\geq n-1})$ given by non-aligned or
  non-standard pairs. As pointed out before, this proposition is crucial, as it deals with
  the dynamics inside the first order critical set $\hcritical1$, and it is, loosely speaking,
  the counterpart of the base equidistribution step (Lemma~\ref{l_baseEquiStep}) for curves
  intersecting $\hcritical1$.
  \begin{prp}\label{l_nonStandard}
  For any $2\leq k \leq n$ and any $(S,n-k)$-compatible standard pair $\ell$, for any function $\averaged\in\continuous{}{\sone}$ with zero average we have:
  \begin{equation}
    \Big|\expectation_{\ell}(\averaged\circ F^k\cdot1_{\tau\geq k-1}) - \sum_{\alpha,j}c_\alpha^j\expectation_{\ell_\alpha^j}(\averaged\circ F^{k-1}\cdot1_{\tau\geq k-2}) \Big|\leq\|\averaged\|\,C_{k} o(\lS^{-1}),\label{e_standBy}
  \end{equation}
  where $c_\alpha^j=\prob_\ell(F^{-1}\ell_\alpha^j)$.
\end{prp}
\begin{proof}
  Since $F\inv\ell_\pm$ does not intersect $\critical{1}$ we have that
  $\prob_\ell(F^{-1}\ell_\pm)=\bigo{\lS^{-\beta}}$, thus, by
  Corollary~\ref{c_lazyInduction}, the contribution of the two non-aligned standard curves
  is $\bigo{\lS^{-2\beta}}=o(\lS^{-1})$. Consequently, we only need to consider the
  contribution of stand-by pairs: if $k>2$ we can conclude by a similar argument:
  decompose the image of stand-by pairs in standard pairs which we denote by
  $\{\ell''_j\}$; then:
  \[
  \sum_{j}\expectation_{\tilde\ell_j}(\averaged\circ F^{k-1}\cdot1_{\tau\geq k-2}) =  \sum_j c''_j\expectation_{\ell''_j}(\averaged\circ F^{k-2}\cdot1_{\tau\geq k-3}),
  \]
  where $c''_j=\prob_\ell(F^{-2}\Gamma''_j)$. Now we can apply
  Corollary~\ref{c_lazyInduction} to the right hand side and conclude, since
  \smash{$\sum_jc''_j\leq\prob_\ell(\hcritical{1})=\bigo{\lS^{-\beta}}$} by
  Lemma~\ref{l_lebesgueCritical}.

  It remains to prove the case $k=2$: we apply the scheme of the proof of
  Lemma~\ref{p_auxBaseStep}; let us denote by $\tilde I$ the base of the preimage of
  a connected component $\tilde\ell$ of the stand-by portion;
  Lemma~\ref{l_lebesgueCritical} implies that we have a uniformly bounded number of
  connected components it suffices to prove that the contribution of each component is
  $o(\lS\inv)$. First of all notice that we necessarily have $\tilde I\subset\{\tau\geq
  1\}$; then, for $\ct{0}\in I$, define $\Theta(\ct{0})=\ct{0}+\T(\cv{0})+\T(\cv{1})$; cut
  $\tilde I$ at the points $\Theta =0\mod 2\pi$ and let $\{J_j\}$ denote the set of
  intervals in $\tilde I$ bounded by two consecutive cutting points. Then applying
  \eqref{e_a3} we immediately obtain:
  \[
  \prob_\ell(\tilde I\setminus\bigcup_k J_k)= \bigo{\lS^{-2\beta}}.
  \]
  Define
  \[
  E_k=\int_{J_k}\averaged(\Theta(x))\rho(x)\deh x;
  \]
   then we can write:
  \[
  \left|\int_{\tilde I}\averaged(\Theta(x))\rho(x)\deh x
  -\sum_kE_k\right|\leq \|\averaged\|o(\V^{-1}).
  \]
  On each $F^2\Gamma_k$ we can define an inverse function for $\Theta$ and we can push forward the density $\rho$ as $\rho''_k(\theta)=\rho(x(\theta))/|\dot\Theta(x(\theta))|$. Separating from the constant part we have:
  \[
  |E_k|\leq 2\pi\|\averaged\|\|\tilde\rho''_k\|\leq 2\pi\|\averaged\|c_k\left\|\de{\log\rho''_k}{\theta}\right\|.
  \]
  The norm can be computed using \eqref{e_doubleIterationR} which gives:
  \[
  \|r''_k\|\leq\underbrace{\frac{6A}{K_2}\left\|\LL_\ell\inv\right\|_{J_k}}_{X_k}+
               \underbrace{\frac{6A}{\bar K_2}\left\|\T'\LL_{\ell'}^{-2}\right\|_{J_k}}_{Y_k}+ \bigo{\lS^{-\beta}}.
  \]
  Again, let us consider the discrete measure space $\Omega$ whose elements are the
  intervals $J_k$, each of measure $c_k$; note that $\mu(\Omega)=\bigo{\lS^{-\beta}}$.
  We then define $f:J_k\mapsto\const X_k$ so that we can apply Lemma~\ref{l_E0} with
  $\lambda=\const\lS^{-\beta}$ and $\alpha=1$; the fact that $f$ satisfies the hypotheses
  of Lemma~\ref{l_E0} follows from the analysis we performed in the proof of
  Lemma~\ref{p_auxBaseStep}. We thus obtain:
  \[
  \sum_k c_k X_k=\bigo{\lS^{-2\beta}\log\lS}=o(\lS\inv).
  \]
  Similarly, let $g:J_k\mapsto\const Y_k$; we claim that the hypotheses of
  Lemma~\ref{l_E0} hold with $\lambda=\const\lS^{-2\beta}$ and $\alpha=1$: this follows
  since if $J_k$ is such that $\T'\LL_{\ell'}^{-2}>\const\lS^{-2\beta}z$, for some $z\geq
  1$, then $\tslope{\ell'}<\const z^{-\onehalf}$, but since
  $\tslope{\ell}\tslope{\ell'}>K_2{\T'}\inv$ we immediately obtain that
  $\LL_\ell=\T'\tslope{\ell}>z^\onehalf$.  We can thus bound the measure of such $J_k$ by
  $\bigo{z\inv}$ i.e. we can choose $\alpha_g=1$. We can then conclude that
    \[
    \sum_kc_k Y_k = o(\V^{-1}),\]
    which implies \eref{e_standBy} and concludes the proof.
\end{proof}
Next, given a clean standard curve $\Gamma$ we want to find a reference curve $\bar\Gamma$ such that the image of $F\bar\Gamma$ shadows $F\Gamma$ very closely; this will be obtained by means of the following
\begin{lemma}[Shadowing by reference curves]\label{l_geometricalReduction}
  Let $\ell$ be a clean standard pair and let $\Gamma^*=\Gamma_\ell\setminus\critical{1}$. Then there exist a reference pair $\bar\ell$ such that $I_{\bar\ell}=I_\ell$ and a subset $\Gamma'\subset\Gamma^*$ such that $\prob_\ell(\Gamma^*\setminus\Gamma')=\bigo{\lS^{-3\beta}}$ and:
  \[
  \fa(\coo{1})\in F\Gamma'\ \ex\bcv{1}\st(\ct{1},\bcv{1})\in F\bar\Gamma\textrm{ and }|\bcv{1}-\cv{1}|=\bigo{\lS^{-5\beta}}.
  \]
\end{lemma}
\begin{proof}
  Define the slope field $\slope{-1}(\coo{0})=-\T'(\cv{0})^{-1}$; then for for any $(\coo{1})\in F\Gamma^*$ consider the vertical line $\{\ct{}=\ct{1}\}$ passing through $(\coo{1})$; it is easy from the definitions to check that $F^{-1}\{\ct{}=\ct{1}\}$ is an integral curve of $\slope{-1}$; moreover, again from the definition it is easy to obtain the relation
  \begin{equation}
    \left.\de{\cv{1}}{\ct{0}}\right|_{\slope{-1}}=\T'(\cv{0})^{-1}.
    \label{e_contraction}
  \end{equation}
  Let $I^*=\pi_x\Gamma^*$ and let $\bar{\ct{}}\in I^*$ such that $|\ddot\phi(\bar{\ct{}})|=\min_{\ct{}\in I^*}|\ddot\phi(x)|$; let $\bar\Gamma=(x,\bar\psi(x))$ be the reference curve over $I$ passing through $(\bar{\ct{}},\psi_\ell(\bar{\ct{}}))$ and $\bar\rho$ be the uniform density on $\bar\Gamma$.  By Lemma~\ref{l_refclose} and the definition of standard curve, we have that the vertical distance between $\Gamma$ and $\bar\Gamma$ is bounded by:
  \begin{equation} \label{e_vertical0}
    |\psi(\ct{})-\bar\psi(\ct{})|\leq \const |\ct{}-\bar{\ct{}}|\lS^{-3\beta}.
\end{equation}
Define
\[
\Gamma'=\{p\in\Gamma^*\st\text{ the integral curve of }\slope{-1}\text{ passing through $p$ intersects }\bar\Gamma\};
\]
then for each $p\in\Gamma'$ we define $\Pi_p$ as the piece of integral curve of $\slope{-1}$ connecting $\Gamma'$ to $\bar\Gamma$. The proof is then complete provided that we prove that $|\pi_x\Pi_{p}|$ is uniformly bounded in $\Gamma'$ by $\const\cdot\lS^{-3\beta}$ and then using \eqref{e_contraction}. First obtain a rough upper bound:
\begin{equation}\label{e_roughHorizontal}
|\pi_x\Pi_{p}|<\frac{\max_{\ct{}\in I^*}|\psi(\ct{})-\psi(\bar{\ct{}})|}{\min_{\ct{}\in I^*}|\slope{\bar\ell}(\ct{})+{\T'}\inv(y_\Pi(x))|}<\const\,{\lS^{-2\beta}}.
\end{equation}
Since $|\tslope{\ell}(\ct{})|>1/2|\tslope{\ell}(\bar{\ct{}})|+\const|\ct{}-\bar{\ct{}}|$, estimates \eref{e_roughHorizontal} and \eref{e_vertical0} allows us to obtain the better estimate:
\[
|\pi_x\Pi_{(\coo{0})}|<2\frac{\max_{|\ct{}-\ct{0}|<\const \lS^{-2\beta}}|\psi(\ct{})-\psi(\bar{\ct{}})|}{\min_{|\ct{}-\ct{0}|<\const \lS^{-2\beta}}\tslope{\ell}(\ct{})}<\const{\lS^{-3\beta}},
\]
which concludes the proof.
\end{proof}
We now proceed to give the proofs of Lemmata~\ref{l_reductionBase}-\ref{l_fourierEstimate}
\begin{proof}[Proof of Lemma~\ref{l_reductionBase}]
  Fix $\bar{\ct{}}\in I$, define $\bar\Gamma=(x,\bar\psi(x))$ to be the reference curve
  over $I$ passing through the point $(\bar{\ct{}},\psi_\ell(\bar{\ct{}}))$, let
  $\bar\rho$ be the uniform density on $I$ and define $\ell^*=(\Gamma_\ell,\bar\rho)$ and
  $\bar\ell=(\bar\Gamma,\bar\rho)$. Then we can write:
  \begin{subequations}
    \begin{align}
      |\expectation_{\vphantom{\bar\ell}\ell}(\averaged\circ
      F)-\expectation_{\bar\ell}(\averaged\circ F)|\leq
      |\expectation_{\vphantom{\bar\ell}\ell}(\averaged\circ
      F)-\expectation_{\vphantom{\bar\ell}\ell^*}(\averaged\circ
      F)|+\label{e_reductionEstimatea}\\\phantom{\leq}+|\expectation_{\vphantom{\bar\ell}\ell^*}(\averaged\circ
      F)-\expectation_{\bar\ell}(\averaged\circ F)|.\label{e_reductionEstimateb}
    \end{align}
  \end{subequations}
  We bound \eqref{e_reductionEstimatea} by applying lemma~\ref{l_baseEquiStep} to $\ell^*$
  with $B=(\rho_\ell-\bar\rho)/\bar\rho$. In fact it is easy to check that $\|B\|\leq
  \delta\|r_\ell\|$ and $\|\dot B\|\leq 2\|r_\ell\|$; hence we obtain
  \begin{equation*}
    |\expectation_\ell(\averaged\circ F)-\expectation_{\ell^*}(\averaged\circ F)|\leq 2\|\averaged\|\|r_\ell\|\const\lS^{-\beta}.
  \end{equation*}
  Introduce the functions $\Theta(\ct{})=\ct{}+\T_\ell(\ct{})$ and
  $\bar\Theta(\ct{})=\ct{}+\T_{\bar\ell}(\ct{})$; Lemma~\ref{l_refclose} implies that
  $|\psi_\ell(\ct{})-\psi_{\bar\ell}(\ct{})|\leq \const\|\Delta\slope{\ell}\||I|$, which
  yields:
  \[
  \|\Theta-\bar\Theta\|\leq \const\,\T'(\V_\ell)\|\Delta\slope{\ell}\|;
  \]
  we can thus rewrite \eqref{e_reductionEstimateb} as:
  \begin{align*}
    |\expectation_{\vphantom{\bar\ell}\ell^*}(\averaged\circ F)&-\expectation_{\bar\ell}(\averaged\circ F)| = \bar\rho\left|\int_I\averaged(\Theta(x))-\averaged(\bar\Theta(x))\deh x\right|\leq \\
    &\leq \bar\rho\,\|\averaged'\|\|\Theta-\bar\Theta\|\leq \const\,\|\averaged\|_1\T'(\V_\ell)\|\Delta\slope{}\|,
  \end{align*}
  which concludes the proof.
\end{proof}
\begin{proof}[Proof of Lemma~\ref{l_periodicityBase}]
  To fix ideas, we consider $\eta_1>\eta_0$ and let $I_\alpha=[a,b]$ assuming without loss of generality that $a=\bar x_\alpha$; moreover introduce the following shorthand notations: $\Gamma_i=\Gamma_{\ell_\alpha^{\eta_i}}$, $\psi_i=\psi_{\ell_\alpha^{\eta_i}}$, $\T_i=\T_{\ell_\alpha^{\eta_i}}$ and similarly for $\T'$ and $\T''$. Define
\begin{align*}
  \Theta_0(\ct{})&=\ct{}+\T_0(\ct{})&
  \Theta_1(\ct{})&=\ct{}+\T_1(\ct{})+\T(\eta_0)-\T(\eta_1)
\end{align*}
so that $\Theta_0(a)=\Theta_1(a)$; let $\delta\Theta=\Theta_1-\Theta_0$ and for $\lambda\in[0,1]$ let $\Theta_\lambda=(1-\lambda)\Theta_0+\lambda\Theta_1$ so that $\partial_\lambda\Theta_\lambda=\delta\Theta$.

We claim that, for any $\lambda\in[0,1]$ the following estimates hold:
\begin{subequations}
  \begin{align}
    |\delta\Theta(x)|&\leq\left|\frac{\dot\Theta_\lambda(x)^2}{\T'(\lS)}\right||\eta_1-\eta_0|\bigo{\lS^{-1}}\label{e_estimateDeltaTheta}\\
    |\delta\dot\Theta(x)|&\leq|\dot\Theta_\lambda(x)||\eta_1-\eta_0|\bigo{\lS^{-1}}\label{e_estimateDotDeltaTheta}
  \end{align}
\end{subequations}
In fact \eqref{e_estimateDotDeltaTheta} follows by direct computations, using the
definition of  $\Theta$; in order to prove \eqref{e_estimateDeltaTheta} write:
\[
\delta\Theta(x)=\int_a^x\delta\dot\Theta(\xi)\deh\xi.
\]
Notice that if $|\dot\Theta_\lambda(x)|>c\T'(\lS)$, then \eqref{e_estimateDeltaTheta} immediately follows from \eqref{e_estimateDotDeltaTheta}; otherwise, we know that $|\ddot\Theta_\lambda(x)|>c'\T'(\lS)$ and that for each $a\leq \xi \leq x$ we have $|\dot\Theta_\lambda(\xi)|\leq|\dot\Theta_\lambda(x)|$; therefore:
\[
|\delta\Theta(x)|\leq |\delta\dot\Theta(x)|\left|\frac{\dot\Theta(x)}{c'\T'(\lS)} \right|
\]
from which we conclude, again using \eqref{e_estimateDotDeltaTheta}.

Define the function
  \begin{equation}\label{e_definitionPsiTilde}
    \tilde\Psi_{\alpha,1}(\lambda)=\int_a^b\averaged(\Theta_\lambda(x))\bar\rho_\alpha\deh\ct{};
  \end{equation}
  clearly $\tilde\Psi_{\alpha,1}(0)=\Psi_{\alpha,1}(\eta_0)$ and since $\T(\eta_1)-\T(\eta_0)=0\mod 1$ we obtain that $\tilde\Psi_{\alpha,1}(1)=\Psi_{\alpha,1}(\eta_1)$. We now claim that:
  \begin{equation}\label{e_claim}
    \de{\tilde\Psi_{\alpha,1}}{\lambda} = \averaged(\Theta_\lambda(b))\frac{\delta\Theta(b)}{\dot\Theta_\lambda(b)}+ \|\averaged\|_1o(\lS^{-1})
  \end{equation}
  from which we conclude; in fact:
  \[
  \Psi_{\alpha,1}(\eta_1)-\Psi_{\alpha,1}(\eta_0)=\int_0^1\de{\tilde\Psi_{\alpha,1}}{\lambda}\deh \lambda=\int_0^1\averaged(\Theta_\lambda(b))\frac{\delta\Theta(b)}{\dot\Theta_\lambda(b)}\bar\rho_\alpha\deh\lambda+ \|\averaged\|_1o(\lS^{-1});
  \]
  notice that by definition $\dot\Theta_\lambda(b)^{-1}$ is a decreasing function of $\lambda$; since $\averaged$ has zero average we obtain the following bound:
  \[
  \left|\int_0^1\averaged(\Theta_\lambda(b))\frac{\delta\Theta(b)}{\dot\Theta_\lambda(b)}\bar\rho_\alpha\deh\lambda\right|<\|\averaged\|\left|\frac{\delta\Theta(b)}{\dot\Theta_0(b)}\right|\bar\lambda,
  \]
  where $\bar\lambda$ is so that $\Theta_{\bar\lambda}(b)=\Theta_0(b)+1$ or $\bar\lambda=1$ if the previous equation has no solutions; since $\partial_\lambda\Theta(b)=\delta\Theta(b)$ we have $\bar\lambda=\min(1,\bigo{\delta\Theta(b)^{-1}})$, which implies
  \[
  \left|\frac{\delta\Theta(b)}{\dot\Theta_0(b)}\right|\bar\lambda\leq \const|\dot\Theta_0(b)^{-1}|=o(\lS^{-1})
  \]
  and concludes the proof. We then need to prove \eqref{e_claim}: differentiate \eqref{e_definitionPsiTilde} with respect to $\lambda$ and obtain
  \[
  \de{\tilde\Psi_{\alpha,1}}{\lambda}=\int_{I_\alpha}\averaged'(\Theta_\lambda(x))\partial_\lambda\Theta_\lambda(x)\bar\rho_\alpha\deh x.
  \]
  Let $J=\{x\in I_\alpha\st |\dot\Theta_\lambda|\geq1\}$; by construction of the partition, one of the boundary points of $J$, which we denote by $a'$ is either equal to $a$ or $\bigo{\lS^{-2\beta}}$-close to $a$ and the other one is necessarily $b$, hence $J=[a',b]$. We have by \eqref{e_estimateDeltaTheta} that $\delta\Theta(a')=\bigo{\lS^{-1-2\beta}}$ and we thus conclude that:
  \[
\int_a^b\averaged'(\Theta_\lambda(x))\partial_\lambda\Theta_\lambda(x)\deh x=\int_{a'}^b\averaged'(\Theta_\lambda(x))\partial_\lambda\Theta_\lambda(x)\deh x+o(\lS^{-1})
  \]
  We then integrate by parts and obtain:
  \begin{align*}
    \int_{a'}^b\averaged'&(\Theta_\lambda(x))\partial_\lambda\Theta_\lambda(x)\deh x =  \int_{a'}^b\averaged'(\Theta_\lambda(x))\dot\Theta_\lambda(x)\frac{\delta\Theta(x)}{\dot\Theta_\lambda(x)}\deh x=\\
    &=\left.\averaged(\Theta_\lambda(x))\frac{\delta\Theta(x)}{\dot\Theta_\lambda(x)}\right|_{a'}^b-\int_{a'}^b\averaged(\Theta_\lambda(x))\left(\frac{\delta\dot\Theta(x)}{\dot\Theta_\lambda(x)}-\frac{\delta\Theta(x)\ddot\Theta_\lambda(x)}{\dot\Theta_\lambda^2(x)}\right)\deh x.
  \end{align*}
  We first deal with the boundary terms; on the one hand the contribution of the term corresponding to $a'$ is $\|\averaged'\|\bigo{\lS^{-1-2\beta}}$ which is negligible; on the other hand the term corresponding to $b$ gives the main term in the right hand side of \eqref{e_claim}. We are thus left to show that the integral term is $o(\lS^{-1})$; let:
  \[
  B=\frac{\delta\dot\Theta}{\dot\Theta_\lambda}-\frac{\delta\Theta\ddot\Theta_\lambda}{\dot\Theta_\lambda^2}.
  \]
  Then we conclude  using Lemma~\ref{p_auxBaseStep} since by  \eqref{e_estimateDeltaTheta} and \eqref{e_estimateDotDeltaTheta} we immediately have that $\|B\|=\bigo{\lS^{-1}}$ and $\|\dot B\|_{I/D}=\bigo{\lS^{-1+\beta}}$.
\end{proof}
\begin{proof}[Proof of Lemma~\ref{l_lipshitzBase}]
  Define $\Theta_\alpha(x,\eta) = x + \T(\psi_{\ell_\alpha^\eta}(x))$; by definition of $\psi_{\ell_\alpha^\eta}(x)$ we can write:
  \begin{align}\label{e_parEtaTheta}
    \partial_\eta\Theta_\alpha(x,\eta)&=\T'(\eta)\frac{\T'(\eta-2\dot\phi(\bar x_\alpha))}{\T'(\eta)}\frac{\T'(\psi_{\ell_\alpha^\eta}(x))}{\T'(\psi_{\ell_\alpha^\eta}(x)-2\dot\phi(x))}
  \end{align}
then by definition:
  \[
  \Psi_{\alpha,1}(\eta)=\int_{I_\alpha}\averaged(\Theta_\alpha(x,\eta))\bar\rho_\alpha\deh x
  \]
  thus:
  \[
  \de{\Psi_{\alpha,1}(\eta)}{\eta}=\int_{I_\alpha}\averaged'(\Theta_\alpha(x,\eta))\partial_\eta\Theta_\alpha(x,\eta)\bar\rho_\alpha\deh x
  \]
  We conclude by applying Lemma~\ref{p_auxBaseStep} to the previous integral using
  $B=\partial_\eta\Theta_\alpha(x,\eta)$; in fact \eqref{e_parEtaTheta} implies that
  $\|\dot B\|=\T'(\eta)(1+\bigo{\lS^{-1}})$ and $\|B\|=\T'(\eta)\bigo{\lS^{-1}}$.
\end{proof}
\begin{proof}[Proof of Lemma~\ref{l_fourierEstimate}]
  As in the proof of Lemma~\ref{l_lipshitzBase}  define $\Theta_{\alpha}(\ct{},\eta)=\ct{}+\T(\psi_{\ell_\alpha^\eta}(x))$; fix $\bar\T=\T(\lS)$ for $\lS\in S$ such that $\bar\T = 0\mod 1$ and notice that
  \[
  \Theta_{\alpha}(\ct{},\eta(\bar\T+\theta))= \Theta_{\alpha}(\ct{},\eta(\bar\T))+\theta+\mu(\theta,x),
  \] with $\|\mu\|_1=\bigo{\lS^{-1}}$  by \eqref{e_parEtaTheta}.
  For $k\in\integers$ define the following sequence:
  \[
  \hat\Psi_{\alpha,1}^{(k)}  =
  \bar\rho_\alpha\int_{I_\alpha} \hat\averaged_ke^{2\pi i k \Theta_{\alpha}(\ct{},\eta(\bar\T))}\deh x
  \]
  Notice that \eref{e_fourierEstimate} follows by applying lemma~\ref{l_baseEquiStep} to the functions $\theta\mapsto\hat\averaged_k e^{2\pi i k\theta}$. Then we need to prove \eqref{e_globalFourierEstimate}. Notice that lemma~\ref{l_periodicityBase} implies that $\Psi_{\alpha,1}(\eta(\T))$ is periodic in $\T$ up to $o(\lS^{-1})$; therefore it suffices to show that \eqref{e_globalFourierEstimate} holds for $\eta\in [\eta(\bar\T),\eta(\bar\T+1)]$. Notice that by definition:
  \[
  \sum_{k\in\integers}\hat\Psi_{\alpha,1} e^{2\pi i\,k \theta}= \bar\rho_\alpha\int_{I_\alpha}\averaged(\Theta_{\alpha}(\ct{},\eta(\bar\T))+\theta)\deh x,
  \]
  and since $\averaged$ is smooth:
  \begin{align*}
    \averaged(\Theta_{\alpha}(\ct{},\eta(\bar\T))+\theta) &=\averaged(\Theta_{\alpha}(\ct{},\eta(\bar\T+\theta)))+\\&\phantom{=} + \averaged'(\Theta_{\alpha}(\ct{},\eta(\bar\T+\theta)))\mu(\theta,x)+\|\averaged''\|\bigo{\mu^2}.
  \end{align*}
  Consequently:
  \begin{align*}
    \Psi_{\alpha,1}(\eta(\bar\T+\theta))- \sum_{k\in\integers}\hat\Psi_{\alpha,1} e^{2\pi i\,k \theta} &= -\int_{I_\alpha}\bar\rho_\alpha\averaged'(\Theta_\alpha(x,\eta(\bar\T+\theta))\mu(\theta,x)\deh x +\\&\phantom{=}+ \|\averaged''\|\bigo{\mu^2}
  \end{align*}
We claim that the right hand side is $o(\lS^{-1})$, which proves \eqref{e_globalFourierEstimate}: in fact by applying Lemma~\ref{p_auxBaseStep} to the first term we obtain a bound $\|\averaged'\|\bigo{\lS^{-1-\beta}\log\lS}$; the second term can in turn be easily bounded since $\mu^2=\bigo{\lS^{-2}}$.
\end{proof}
We will prove Lemmata~\ref{l_lipshitz}-\ref{l_periodicity} by means of the following induction scheme: using Lemmata~\ref{l_reductionBase}-\ref{l_fourierEstimate}  we will prove Lemma \ref{l_baseInductionStep} for $k=2$, from which will follow Lemmata~\ref{l_lipshitz}-\ref{l_periodicity} for $k=2$; then assuming we proved Lemmata~\ref{l_lipshitz}-\ref{l_periodicity} for $k$, we prove \ref{l_baseInductionStep} for $k+1$ and thus Lemmata~\ref{l_lipshitz}-\ref{l_periodicity} for $k+1$.

The following Lemma is the base induction step which will be used in all the remaining proofs.
\begin{prp}[Base induction step]\label{l_baseInductionStep}
  For all $2\leq k \leq n$ we have:
  \begin{enumerate}[(a)]
  \item let $\ell$ be any clean standard pair that is $(S,n-k)$-compatible; then by proposition \ref{l_nonStandard} we know that:
    \begin{align*}
      \expectation_{\ell}(\averaged\circ F^k\cdot1_{\tau\geq k-1}) &= \sum_{\alpha,j}c_\alpha^j\expectation_{\ell_\alpha^j}(\averaged\circ F^{k-1}\cdot1_{\tau\geq k-2}) + \|\averaged\|o(\lS^{-1}).
    \end{align*}
    For each $\alpha$ there exists an index set $\Kset{\alpha}$, which excludes at most a
    uniformly bounded number of indices, and $\{\eta_\alpha^j\}$ such that
    \begin{align}
    \expectation_\ell(\averaged\circ F^k\cdot 1_{\tau\geq k-1})=
    \sum_{\alpha}\sum_{j\in\Kset{\alpha}} c_\alpha^j\Psi_{\alpha,k-1}(\eta_\alpha^j)  +\|\averaged\| o(\lS^{-1}),\label{e_baseInductionOneStar}
  \end{align}
  where  $\eta_\alpha^j$ satisfies the following estimate:
    \begin{equation}\label{e_etaDistribution}
      \T(\eta_\alpha^j)=\T(\eta_\alpha^0)+j-\theta_\alpha^j+\theta_\alpha^0+\mu_\alpha(j-\theta_\alpha^j+\theta_\alpha^0)
    \end{equation}
    where $\theta_\alpha^j\in I$ belongs to a $\bigo{\lS^{-3\beta}}$-neighborhood of
    the point $\bar\theta_{\alpha}^j$, where $\bar\theta_{\alpha}^j$ is such that
    \[\pi_xF(\bar\theta_{\alpha}^j,\psi(\bar\theta_{\alpha}^j))=\bar x_{\alpha}\]
    and $\mu_\alpha$ is a function such that:
    \begin{align}\label{e_conditionsNu}
      \mu_\alpha(0)&=0&\mu'_\alpha&=\left(\frac{\T'(\lS+2\dot\phi(\bar{\ct{}}_{\alpha}))}{\T'(\lS)}-1\right)+\bigo{\lS^{-2}}&\mu''_\alpha&=\bigo{\lS^{-2-2\beta}}.
    \end{align}
\item for any two clean standard pairs $\ell_1$ and $\ell_2$, both $(S,n-k)$-compatible
  and such that  $\|\T\circ\psi_1-\T\circ\psi_2\|<1$, then for each $\alpha$ there exists a
  common index set $\Kset\alpha$ satisfying:
    \begin{align*}
      \expectation_{\ell_1}(\averaged\circ F^k\cdot 1_{\tau\geq k-1})&=\sum_{\alpha}\sum_{j\in\Kset{\alpha}}c_{\alpha,1}^j\Psi_{\alpha,k-1}(\eta_{\alpha,1}^j)  +\|\averaged'\| o(\lS^{-1})\\
      \expectation_{\ell_2}(\averaged\circ F^k\cdot 1_{\tau\geq k-1})&=\sum_{\alpha}\sum_{j\in\Kset{\alpha}}c_{\alpha,2}^j\Psi_{\alpha,k-1}(\eta_{\alpha,2}^j)  +\|\averaged'\| o(\lS^{-1})
    \end{align*}
    and the following estimate holds true:
    \begin{equation}\label{e_etaCloseness}
      |\eta_{\alpha,1}^j-\eta_{\alpha,2}^j|\leq 2\frac{\|\psi_1-\psi_2\|}{\tslope{}(\theta_{\alpha,1}^j)\T'}.
    \end{equation}
  \end{enumerate}
\end{prp}

\begin{proof}
    By proposition~\ref{l_nonStandard} we know that
    \[ \expectation_\ell(\averaged\circ F^k\cdot 1_{\tau\geq k-1})=\sum_{\alpha,j}
    c_\alpha^j\expectation_{\ell_\alpha^j}(\averaged\circ F^{k-1}\cdot 1_{\tau\geq
      k-2})+\|\averaged\| o(\lS^{-1}).\label{e_baseInductionTwoStars}
    \]
    In order to obtain \eref{e_baseInductionOneStar}, consider the reference curve
    $\bar\Gamma$ with $\bar\Gamma=(\ct{},\bar\psi(\ct{}))$ given by
    Lemma~\ref{l_geometricalReduction}. 
    We know that its image is $\bigo{\lS^{-5\beta}}$-close to $F\Gamma^*$ along the
    vertical direction outside a small set of measure $\bigo{\lS^{-3\beta}}$ which we 
    neglect. Hence we can find $\eta_\alpha^k$s such that
    \[
    (\bar{\ct{}}_\alpha,\eta_\alpha^k)\in F\bar\Gamma\textrm{ is close to }\Gamma_\alpha^k.
    \]
    First we prove equation \eref{e_etaDistribution}: by definition, a point $(\coo{})$ is in the preimage of $\{\ct{}=\bar{\ct{}}_\alpha\}$ if it satisfies the following equation:
    \begin{equation}\label{e_varthetaDefinition}
      \ct{}+\T(\cv{})=\bar{\ct{}}_\alpha\mod 1
    \end{equation}
    Therefore, by imposing $(\coo{})\in\bar\Gamma$ we obtain an equation for the points
    \[
    \theta_\alpha^k=\pi_x F^{-1}(\bar{\ct{}}_\alpha,\eta_\alpha^k);
    \]
    since $\ell$ is a clean standard pair, we can write:
    \[
    \T(\bar\psi(\theta_\alpha^k))=\bar{\ct{}}_\alpha+N_\alpha+k-\theta_\alpha^k,
    \]
    where $N_\alpha$ is such that $\theta_\alpha^{k=0}$ is the point satisfying
    \eref{e_varthetaDefinition} closest to $\bar{\ct{}}_\alpha$ and; we define
    $\Kset\alpha$ as the set of $k$'s which satisfy the above equation. Notice that since $\ell$ is a clean standard pair we have either $\Kset\alpha\subset\{k\leq 0\}$ or $\Kset\alpha\subset\{k\geq 0\}$. Since we have
    \[
    \eta_\alpha^k=\bar\psi(\theta_{\alpha}^k)+2\dot\phi(\bar{\ct{}}_\alpha),
    \]
    we define $\mu_\alpha$ such that the following equation holds true:
    \begin{align*}
      \T(\eta_\alpha^k)&=\T(\T^{-1}(\bar{\ct{}}_\alpha+N_\alpha+k-\theta_\alpha^k)+2\dot\phi(\bar{\ct{}}_\alpha))=\\
      &=\T(\eta_\alpha^0)+k-\theta_\alpha^k+\theta_\alpha^0+\mu_\alpha(k-\theta_\alpha^k+\theta_\alpha^0).
    \end{align*}
    Clearly $\mu_\alpha(0)=0$; by simple calculations we obtain:
    \begin{align*}
      \mu_\alpha'(t)&=\frac{\T'(\T^{-1}(\bar{\ct{}}_\alpha+N_\alpha+t-\theta_\alpha^0)+2\dot\phi(\bar{\ct{}}_\alpha)}{\T'(\T^{-1}(\bar{\ct{}}_\alpha+N_\alpha+t-\theta_\alpha^0)}\\
      \mu_\alpha''(t)&=\bigo{\lS^{-2-2\beta}},
    \end{align*}
      which imply \eref{e_conditionsNu}.
      We now need to estimate $|\expectation_{\ell_\alpha^k}(\averaged\circ F^{n-1}\cdot 1_{\tau\geq n-2})-\Psi_{\alpha,n-1}(\eta_\alpha^k)|$ for all $\alpha\in\aset$ and $k\in\Kset{\alpha}$.
    Consider the case $n=2$; we use Lemmata~\ref{l_reductionBase} and~\ref{l_lipshitzBase} which yield:
    \begin{align*}
      |\expectation_{\ell_\alpha^k}(\averaged\circ F)&-\Psi_{\alpha,1}(\eta_\alpha^k)|\leq\\
      &\leq C\|\averaged\|_1(\|r\|_{\ell_\alpha^k}\lS^{-\beta}+\T'(\lS)\|\Delta\slope{}\|_{\ell_\alpha^k}+\|\Delta\T\|_{\ell_\alpha^k}\lS^{-\beta}).
    \end{align*}
    The last term on the right hand side is $\bigo{\lS^{-4\beta}}$ by Lemma~\ref{l_geometricalReduction}; in order to obtain \eref{e_baseInductionOneStar} we are left with obtaining a bound for the following quantity:
    \begin{equation}\label{e_boundn2}
      \sum_{k\in\Kset{\alpha}} c_\alpha^k C\|\averaged\|_1(\|r\|_{\ell_\alpha^k}\lS^{-\beta}+\T'(\lS)\|\Delta\slope{}\|_{\ell_\alpha^k}).
    \end{equation}
    Applying Lemma~\ref{l_E0} gives a bound $\bigo{\lS^{-2\beta}\log\lS}$ which gives \eref{e_baseInductionOneStar} and completes the proof of (a) for $n=2$

    Consider now the case $n\geq 3$; we assume by inductive hypothesis that Lemmata~\ref{l_reduction} and~\ref{l_lipshitz} hold for step $(n-1)$. We therefore obtain:
    \begin{align*}
      |\expectation_{\ell_\alpha^k}(\averaged\circ F^{n-1}&\cdot 1_{\tau\geq n-2})-\Psi_{\alpha,n-1}(\eta_\alpha^k)|\leq\\
      &\leq\|r\|_{\ell_\alpha^k}\expectation_{\ell_\alpha^k}(\averaged\circ F^{n-1}\cdot 1_{\tau\geq n-2})+\|\averaged\|_1 C_{n}\lS^{-2\beta}\log\lS.
  \end{align*}
  Using once more Lemma~\ref{l_E0} we estimate the sum $\sum_k
  c_\alpha^k\|r\|_{\ell_\alpha^k}=\bigo{\lS^{-\beta}}$; applying
  corollary~\ref{c_lazyInduction} gives \eref{e_baseInductionOneStar} and concludes the
  proof of (a) in the general case.

  In order to prove part (b), it suffices to apply part (a) to both pairs; since $\ell_1$
  and $\ell_2$ are close to each other, we can adjust $N_\alpha$ and $\Kset\alpha$ of a
  bounded quantity in order to find a set of indices which is common to both $\ell_1$ and
  $\ell_2$. In doing this we discard at most an uniformly bounded number of standard
  pairs $\ell_\alpha^k$, which contribute at most with
  $\|\averaged\|\bigo{\lS^{-2\beta}\log\lS}$ and can therefore be neglected.  Estimate
  \eref{e_etaCloseness} then follows by simple geometrical considerations similar to those
  used in the proof of Lemma~\ref{l_geometricalReduction}.
\end{proof}
We now conclude this section by proving Lemmata~\ref{l_lipshitz}-\ref{l_periodicity}.
\begin{proof}[Proof of Lemma~\ref{l_lipshitz}]
  Lemma~\ref{l_baseInductionStep}b allows to write the following estimate:
  \begin{align*}
    &|\Psi_{\alpha,k}(\eta_1)-\Psi_{\alpha,k}(\eta_2)|\leq\notag\\&\leq\sum_{\alpha'}\sum_{j\in\Kset{{\alpha'}}}|c_{\alpha',1}^j\Psi_{\alpha',k-1}(\eta_{\alpha',1}^j)-c_{\alpha',2}^j\Psi_{\alpha',k-1}(\eta_{\alpha',2}^j)|+
    \|\averaged\|_1 o(\lS^{-1}).
  \end{align*}
  We now claim that, for any given $\alpha'$ we have:
  \begin{equation}\label{e_estimateSingleAlpha}
    \sum_{j\in\Kset{{\alpha'}}}|c_{\alpha',1}^j\Psi_{\alpha',k-1}(\eta_{\alpha',1}^j)-c_{\alpha',2}^j\Psi_{\alpha',k-1}(\eta_{\alpha',2}^j)|=\|\averaged\|_1o(\lS\inv).
  \end{equation}
  In order to estimate each term of the sum we write:
  \begin{subequations}
    \begin{align}
      |c_{\alpha',1}^j\Psi_{\alpha',k-1}(\eta_{\alpha',1}^j)&-c_{\alpha',2}^j\Psi_{\alpha',k-1}(\eta_{\alpha',2}^j)|=\notag\\&=|(c_{\alpha',1}^j-c_{\alpha',2}^j)\Psi_{\alpha',k-1}(\eta_{\alpha',1}^j)|+\label{e_simpleTerma}\\ &\phantom{=}+ c_{\alpha',2}^j|\Psi_{\alpha',k-1}(\eta_{\alpha',1}^j)-\Psi_{\alpha',k-1}(\eta_{\alpha',2}^j)|.\label{e_simpleTermb}
    \end{align}
  \end{subequations}
We obtain a bound for \eref{e_simpleTerma} in the following way: define $\Theta_i$ as in
the proof of Lemma~\ref{l_lipshitzBase} and, for fixed $j$ and $\alpha'$ let
$\iTheta_i$ be the inverse function of $\Theta_i$ on $I_{\alpha'}$; then we can estimate:
  \begin{align*}
    |c_{\alpha',1}^j-c_{\alpha',2}^j|&=\bar\rho_\alpha\int_{I_{\alpha'}}\left|\frac{1}{\dot\Theta_1(\xi_1(\theta))}-\frac{1}{\dot\Theta_2(\xi_2(\theta))}\right|\deh \theta=\\
    &=\bar\rho_\alpha\int_{I_{\alpha'}} \frac{1}{\dot\Theta_1(\xi_1(\theta))}\left|1-\frac{\dot\Theta_1(\xi_1(\theta))}{\dot\Theta_2(\xi_2(\theta))}\right|\deh \theta=\\
    &\leq c_{\alpha',1}^j\left\|1-\frac{\dot\Theta_1(\xi_1(\theta))}{\dot\Theta_2(\xi_2(\theta))}\right\|_{\alpha'}^j.
  \end{align*}
  By simple geometrical considerations we obtain:
  \[
  \left\|1-\frac{\dot\Theta_1(\xi_1)}{\dot\Theta_2(\xi_2)}\right\|_{\alpha'}^j\leq\left\|\frac{\ddot\Theta_2}{\dot\Theta_2^2}\right\|_{\alpha'}^j\cdot|\T(\eta_1)-\T(\eta_2)|.
  \]
  from which we conclude, using once more Lemma~\ref{l_E0} and
  corollary~\ref{c_lazyInduction}, that the sum over \eref{e_simpleTerma} contributes with
  $o(\lS^{-1})$.

  In order to bound \eref{e_simpleTermb}, first assume that $k\geq 3$ and that we proved \eref{e_lipshitz} at step $(k-1)$; we then apply estimate \eref{e_etaCloseness} and conclude by Lemma~\ref{l_lipshitz} at step $k-1$ that the contribution of \eref{e_simpleTermb} is bounded by $o(\lS^{-1}$ as well.
  For the base case $k=2$ we need to use Lemma~\ref{l_lipshitzBase}, which yields the following bound for the contribution of \eref{e_simpleTermb}:
  \[
  \sum_{j\in\Kset{\alpha}}c_{\alpha',2}^j C\|\averaged\|_1|\T(\eta_{\alpha,1}^j)-\T(\eta_{\alpha,2}^j)|\lS^{-\beta}.
  \]
  Using once again Lemma~\ref{l_E0} and \eref{e_etaCloseness} we obtain a bound of order $o(\lS\inv)$, which concludes the proof.
\end{proof}
\begin{proof}[Proof of Lemma~\ref{l_reduction}]
   We apply Lemma~\ref{l_geometricalReduction} in order to obtain a reference pair
   $\bar\ell$; we claim that $\bar\ell$ satisfies \eref{e_reductionEstimate}. Define
   $\ell^*$ as we did in Lemma~\ref{l_reductionBase}, i.e. $\ell^*=(\Gamma,\rho^*)$
   where $\rho^*$ is the uniform density on $I$. Then, as in
   Lemma~\ref{l_reductionBase}, we need to estimate the quantity
   $|\expectation_\ell-\expectation_{\ell^*}|+|\expectation_{\ell^*}-\expectation_{\bar\ell}|$;
   by proposition~\ref{l_nonStandard} we can neglect the contribution of curves inside
   $\critical1$, thus, by Lemma~\ref{l_baseInductionStep}:
  \begin{align*}
    |\expectation_\ell(\averaged\circ F^k\cdot 1_{\tau\geq k-1})&-  \expectation_{\ell^*}(\averaged\circ F^k\cdot 1_{\tau\geq k-1})|=\\&=\sum_{\alpha\in\aset}\sum_{j\in\Kset\alpha}|c_{\alpha}^j-c_{\alpha}^{j*}||\Psi_{\alpha,k-1}(\eta_{\alpha}^j)|+\|\averaged\|o(\lS^{-1}).
  \end{align*}
  It is not difficult to check that
  \[
  |c_{\alpha}^j-c_{\alpha}^{j*}|\leq \const\|r\|c_{\alpha}^{j*},
  \]
  which implies that
  \begin{align*}
  |\expectation_\ell(\averaged\circ F^k\cdot 1_{\tau\geq k-1})&-
  \expectation_{\ell^*}(\averaged\circ F^k\cdot 1_{\tau\geq
    k-1})|\leq\\&\const\|r\||\expectation_{\ell^*}(\averaged\circ F^k\cdot 1_{\tau\geq
    k-1})|+\|\averaged\|o(\lS^{-1}).
\end{align*}
We use once more Lemma~\ref{l_baseInductionStep}b on $\ell^*$ and $\bar\ell$ noticing that, by construction, the $\eta_\alpha^j$ appearing in their respective \eref{e_baseInductionOneStar} coincide; hence:
    \begin{align}
      |\expectation_{\ell^*}(\averaged\circ F^k\cdot 1_{\tau\geq k-1})&-  \expectation_{\bar\ell}(\averaged\circ F^k\cdot 1_{\tau\geq k-1})|\leq\notag\\
      &\leq\sum_{\alpha\in\aset}\sum_{j\in\Kset\alpha} |c_{\alpha}^{j*}-\bar
      c_{\alpha}^{j}||\Psi_{\alpha,k-1}(\eta_{\alpha}^{j})|+\|\averaged\|o(\lS^{-1}).
      \label{e_reductionStara}
    \end{align}
    We can estimate \eref{e_reductionStara} in the following way:
  \begin{align*}
    |c_{\alpha}^{j*}-\bar c_{\alpha}^{j}|&\leq\int_{I_\alpha}\bar\rho\left|\frac{1}{\dot\Theta(\ct{}^*(\theta))}-\frac{1}{\dot{\bar\Theta}(\bar{\ct{}}(\theta))}\right|\deh\theta\leq\\
    &\leq c_\alpha^{j*}\left\|1-\frac{\dot\Theta(\ct{}^*(\theta))}{\dot{\bar\Theta}(\bar{\ct{}}(\theta))}\right\|.
  \end{align*}
  By construction of $\bar\ell$, we know that $|\bar{\ct{}}-\ct{}^*|=\bigo{\lS^{-3\beta}}$, hence from the definition of standard curve, and from the fact that we consider points outside $\critical{1}$ we have:
  \[
  \left\|1-\frac{\dot\Theta(\ct{}^*(\theta))}{\dot{\bar\Theta}(\bar{\ct{}}(\theta))}\right\|\leq\bigo{\lS^{-2\beta}},
  \]
  which concludes the proof of \eref{e_reductionEstimate}.
\end{proof}
\begin{proof}[Proof of Lemma~\ref{l_periodicity}]
  First of all we use Lemma~\ref{l_baseInductionStep} to obtain:
  \[
  \Psi_{\alpha,k}(\eta)=\sum_{\alpha'\in\aset}\sum_{j\in\Kset{{\alpha'}}}c_{\alpha'}^j(\eta)\Psi_{\alpha',k-1}(\eta_{\alpha'}^j(\eta))+\|\averaged\| o(\lS\inv);
  \]
  for each $\alpha'\in\aset$, let $\eta^*(\eta,\alpha,\alpha')$ such that
  $|\T(\eta^*)-\T(\eta)|<1$ and $\pi_x
  F(\bar{\ct{}}_\alpha,\eta^*)=\bar{\ct{}}_{\alpha'}$. By \eqref{e_estimateSingleAlpha} we
  conclude that:
  \begin{align*}
    \sum_{j\in\Kset{{\alpha'}}}c_{\alpha'}^j(\eta)&\Psi_{\alpha',k-1}(\eta_{\alpha'}^j(\eta))=\\
    &=\sum_{j\in\Kset{{\alpha'}}}c_{\alpha'}^j(\eta^*)\Psi_{\alpha',k-1}(\eta_{\alpha'}^j(\eta^*))+\|\averaged\|_1 o(\lS\inv).
  \end{align*}
We claim that we can neglect the dependency of $c_{\alpha'}^j$ on $\eta^*$; in fact, if $\eta^*_1,\eta^*_2$ are such that $\pi_x(F(\bar{\ct{}}_\alpha,\eta_i^*))=\bar{\ct{}}_{\alpha'}$, then necessarily $\T(\eta^*_1)\equiv\T(\eta^*_2)\mod 1$. Consequently, following the proof of Lemma~\ref{l_periodicityBase} we obtain:
  \[
  |c_{\alpha'}^j(\eta^*_1)- c_{\alpha'}^j(\eta^*_2)|\leq C|\eta^*_1-\eta^*_2|\lS^{-1}\cdot c_{\alpha'}^j(\eta_1^*).
  \]
  Since $\Psi_{\alpha',k-1}=\|\averaged\|\bigo{\lS^{-\beta}}$ by corollary~\ref{c_lazyInduction}, we conclude that we can find a sequence $c_{\alpha'}^j$ such that:
  \[
  \sum_{j\in\Kset{{\alpha'}}}c_{\alpha'}^j(\eta^*)\Psi_{\alpha',k-1}(\eta_{\alpha'}^j(\eta^*))=
  \sum_{j\in\Kset{{\alpha'}}}c_{\alpha'}^j\Psi_{\alpha',k-1}(\eta_{\alpha'}^j(\eta^*))+\|\averaged\|_1 o(\lS\inv),
  \]
  that is:
  \[
  \Psi_{\alpha,k}(\eta)=  \sum_{j\in\Kset{{\alpha'}}}c_{\alpha'}^j\Psi_{\alpha',k-1}(\eta_{\alpha'}^j(\eta^*))+\|\averaged\|_1 o(\lS\inv).
  \]
  We now consider separately the cases $k=2$ and $k\geq3$; we first assume that $k=2$: in
  this case, Lemma~\ref{e_baseEquiStep} implies that
  $\Psi_{\alpha',1}=\|\averaged\|o(\lS\inv)$ for $\alpha'$ corresponding to curves which
  do not intersect the critical set $\hcritical1$; we can therefore neglect the
  contribution of such curves. We let $\asets$ be the subset of $\aset$ given by indices
  associated to the remaining curves.  We use Lemma~\ref{l_fourierEstimate} and
  \eref{e_etaDistribution} to obtain:
  \begin{align*}
    &\sum_{j\in\Kset{{\alpha'}}}c_{\alpha'}^j\Psi_{\alpha',k-1}(\eta_{\alpha'}^j(\eta^*))=\\
    &=\sum_{j\in\Kset{{\alpha'}}}c_{\alpha'}^j{\sum_l}'\hat\Psi_{\alpha',1}^{(l)}e^{-2\pi i\,l\T(\eta_{\alpha'}^j(\eta^*))}=\\
    &={\sum_l}\hat\Psi_{\alpha',1}^{(l)}e^{-2\pi i\,l\T(\eta_{\alpha'}^0(\eta^{*}))}\underbrace{\sum_{j\in\Kset{{\alpha'}}}c_{\alpha'}^je^{-2\pi i\,l(\mu_{\alpha'}(j-\theta_{\alpha'}^j+\theta_{\alpha'}^0)-\theta_{\alpha'}^j+\theta_{\alpha'}^0)}}_{\Upsilon^{(l)}_{\alpha'}}.
  \end{align*}
  We can now understand the cancellation mechanism: we claim that
\begin{equation}\label{e_estimateUpsilon1}
  |\Upsilon_{\alpha'}^{(l)}|\leq C\lS^{1/2-\beta}+\min(C\,|l|\lS^{1-2\beta}\log\lS+C\,|l|\lS^{-1},3);
\end{equation}
In fact $\Upsilon_{\alpha'}^{(l)}$ is given by a sum of oscillating terms; the phase of
each term differs from the phase of the previous one by $\bigo{\lS^{-1}}$, and we have
$\bigo{\lS^{2\beta}}$ such terms. We will collect together the phases belonging to the
same period: in an ideal (unrealistic) situation, standard pairs would have uniform
weights, and therefore summing over each complete collection would give us a contribution
of order $\lS^{-1}$, by comparison with a Riemann sum.  In reality standard pairs have
non-uniform weights and we need more involved estimates in order to deal with the lack of
uniformity of weights.

If $\baseSlope{\bar{\ct{}}_{\alpha}}=0$ it is necessary to avoid a
portion of curve where $|\tslope{1}|$ is too small; namely, define
$\Theta(\ct{})=\ct{}+\T(\psi_{\alpha,\eta}(\ct{}))$ and let $\ct{}^*(\alpha,\alpha')\in
I_\alpha$, such that
$|\mu_{\alpha'}(\Theta(\ct{}^*))-\mu_{\alpha'}(\Theta(\bar{\ct{}}_{\alpha}))|=2$; let
$\slope{}^*=|\baseSlope{\ct{}^*}|$ and define $\kset{\alpha'}\subset \Kset{\alpha'}$ such
that $\fa j\in\kset{\alpha'}$ we have $|\tslope{}(\theta^j_{\alpha'})|\geq \slope{}^*$. It
is easy to prove, by definition of $\mu_\alpha'$ and $\slope{}$, that
$|\bar{\ct{}}_\alpha-\ct{}^*|\leq\const\cdot\lS^{\onehalf-\beta}$ and
$\slope{}\geq\const\lS^{\onehalf-\beta}$, which immediately implies that
$\sum_{j\in\Kset{\alpha'}\setminus\kset{\alpha'}}c_{\alpha'}^j\leq\const\cdot\lS^{\onehalf-\beta}$. On
the other hand, if $\baseSlope{\bar{\ct{}}_\alpha}\not=0$ we can simply take
$\kset{\alpha'}=\Kset{\alpha'}$. Consider a partition $\kset{\alpha'}$ in subsets
$\prt{\kset{\alpha'}}{i}$ such that $j,j'\in\prt{\kset{\alpha'}}{i}$ if and only if:
\[
\lfloor\mu_{\alpha'}(j-\theta_{\alpha'}^j+\theta_{\alpha'}^0)\rfloor = \lfloor\mu_{\alpha'}(j'-\theta_{\alpha'}^{j'}+\theta_{\alpha'}^0)\rfloor.\]
Denote $\prt{\dot\Theta}{i}=\min_{j\in\prt{\kset{\alpha'}}{i}}|\baseSlope{\theta_{\alpha'}^j}\T'(\lS)|$; \eref{e_conditionsNu} then implies:
\begin{equation}\label{e_maxTheta}
\max_{j,j'\in\prt{\kset{\alpha'}}{i}}(\theta_{\alpha'}^j-\theta_{\alpha'}^{j'})\leq \const\cdot\lS\prt{\dot\Theta}{i}^{-1}=\bigo{\lS^{\onehalf-\beta}}
\end{equation}
Define $\prt{c_{\alpha'}}{i}=\sum_{j\in\prt{\kset{\alpha'}}{i}}c_{\alpha'}^j$; notice that \eref{e_maxTheta} implies
\begin{equation}\label{e_maxC}
  \prt{c_{\alpha'}^j}{i}\leq \const\,\lS \prt{\dot\Theta}{i}^{-1},
\end{equation}
  and that, for any $j,j'\in\prt{\kset{\alpha'}}{i}$:
\begin{align}\label{e_variationTheta}
  |e^{-2\pi i\,l\mu_{\alpha'}(j-\theta_{\alpha'}^j+\theta_{\alpha'}^0)-\theta_{\alpha'}^j+\theta_{\alpha'}^0}&-e^{-2\pi i\,l\mu_{\alpha'}(j-\theta_{\alpha'}^{j'}+\theta_{\alpha'}^0)-\theta_{\alpha'}^{j'}+\theta_{\alpha'}^0}|\notag\\&\leq \min(\const\,|l|\lS\prt{\dot\Theta}{i}^{-1},2).
\end{align}
We want to keep those $\prt{\kset{\alpha'}}{i}$ for which $\{\mu_{\alpha'}(j-\theta_{\alpha'}^j+\theta_{\alpha'}^0)\mod 1\}_{j\in\prt{\kset{\alpha'}}{i}}$ samples $\sone$ with an error bounded by $\bigo{\lS^{-1}}$. It is sufficient to discard the \emph{first} and \emph{last} (with respect to the natural ordering given by $j$) of the $\prt{\kset{\alpha'}}{i}$; in fact, by \eref{e_maxC}, their contribution is bounded by $\const\,\lS^{\onehalf-\beta}$. Define $\bar j_i=\min\prt{\kset{\alpha'}}{i}$ and compute the following sum:
\[
\sum_{j\in\prt{\kset{\alpha'}}{i}}\frac{1}{|\prt{\kset{\alpha'}}{i}|}e^{-2\pi i\,l\mu_{\alpha'}(j-\theta_{\alpha'}^{\bar j_i}+\theta_{\alpha'}^0)};
\]
by \eref{e_conditionsNu}, $\mu_{\alpha'}$ is almost a linear function in each $\prt{\kset{\alpha'}}{i}$; in particular, $\fa j\in\prt{\kset{\alpha'}}{i}$ the following bound holds true:
\[
\mu_{\alpha'}(j-\theta_{\alpha'}^{\bar j_i}+\theta_{\alpha'}^0)=\mu_{\alpha'}(\bar j_i-\theta_{\alpha'}^{\bar j_i}+\theta_{\alpha'}^0)+\mu'_{\alpha'}(\bar j_i-\theta_{\alpha'}^{\bar j_i}+\theta_{\alpha'}^0)(j-\bar j_i)+\bigo{\lS^{-1-2\beta}}.
\]
where $\mu'_{\alpha'}$ is the derivative of $\mu_{\alpha'}$. By hypothesis $\hat\Psi_{\alpha,1}^{(0)}=0$, hence we can assume $l\not =0$; comparison with the Riemann sum of $\int e^{2\pi i l \theta}\deh\theta$ implies:
\begin{equation}\label{e_RiemannSum}
\sum_{j\in\prt{\kset{\alpha'}}{i}}\frac{1}{|\prt{\kset{\alpha'}}{i}|}e^{-2\pi i\,l\mu_{\alpha'}(j-\theta_{\alpha'}^{\bar j_i}+\theta_{\alpha'}^0)}\leq \min(\const\ |l|\lS^{-1},1).
\end{equation}
Finally, by definition of $c_{\alpha'}^j$ we have the following estimate:
\[
\left|\frac{c_{\alpha'}^j}{\prt{c_{\alpha'}}{i}}-\frac{1}{|\prt{\kset{\alpha'}}{i}|}\right|\leq\const\cdot\lS^{2\beta}\prt{\dot\Theta}{i}^{-2}.
\]
We can therefore estimate $|\Upsilon_{\alpha'}^{(l)}|$ as follows:
\begin{subequations}
\begin{align}
  |\Upsilon_{\alpha'}^{(l)}|\leq\sum_i\prt{c_{\alpha'}}{i}&\sum_{j\in\prt{\kset{\alpha'}}{i}}\left|\frac{c_{\alpha'}^j}{\prt{c_{\alpha'}}{i}}-\frac{1}{|\prt{\kset{\alpha'}}{i}|}\right|+\label{e_estimateUpsilona}\\
  +&\left|\sum_{j\in\prt{\kset{\alpha'}}{i}}\frac{1}{|\prt{\kset{\alpha'}}{i}|}(e^{-2\pi i\,l\mu_{\alpha'}(j-\theta_{\alpha'}^j+\theta_{\alpha'}^0)-\theta_{\alpha'}^j+\theta_{\alpha'}^0}+\right.\notag\\&\left.\phantom{\sum_{j\in\prt{\kset{\alpha'}}{i}}\frac{1}{|\prt{\kset{\alpha'}}{i}|}} -e^{-2\pi i\,l\mu_{\alpha'}(j-\theta_{\alpha'}^{\bar{j}_i}+\theta_{\alpha'}^0)-\theta_{\alpha'}^{\bar{j}_i}+\theta_{\alpha'}^0})\right|+\label{e_estimateUpsilonb}\\
  +&\left|\sum_{j\in\prt{\kset{\alpha'}}{i}}\frac{1}{|\prt{\kset{\alpha'}}{i}|}e^{-2\pi i\,l\mu_{\alpha'}(j-\theta_{\alpha'}^{\bar j_i}+\theta_{\alpha'}^0)-\theta_{\alpha'}^{\bar j_i}+\theta_{\alpha'}^0}\right|.\label{e_estimateUpsilonc}
\end{align}
\end{subequations}
The sum in \eref{e_estimateUpsilona} can be bounded using Lemma~\ref{l_E0}; consider the
finite measure space whose elements are the subsets $\prt{\kset{\alpha'}}{i}$ with measure
$\prt{c_{\alpha'}}{i}$; we let
\[
f_i=\sum_{j\in\prt{\kset{\alpha'}}{i}}\left|\frac{c_{\alpha'}^j}{\prt{c_{\alpha'}}{i}}-\frac{1}{|\prt{\kset{\alpha'}}{i}|}\right|\leq\const\cdot\lS^{1+2\beta}\prt{\dot\Theta}{i}^{-2}
\]
and
\begin{align*}
\lambda&=\const\cdot\lS^{1-2\beta}& \alpha&=1/2.
\end{align*}
This gives a bound $\bigo{\lS^{\onehalf-\beta}}$; the sum in
\eref{e_estimateUpsilonb} can be bounded again by Lemma~\ref{l_E0}; this time let
\begin{align*}
  X_i=\left|\sum_{j\in\prt{\kset{\alpha'}}{i}}\right.&\frac{1}{|\prt{\kset{\alpha'}}{i}|}(e^{-2\pi i\,l\mu_{\alpha'}(j-\theta_{\alpha'}^j+\theta_{\alpha'}^0)-\theta_{\alpha'}^j+\theta_{\alpha'}^0}+\\-&\left.\vphantom{\sum_{j\in\prt{\kset{\alpha'}}{i}}\frac{1}{|\prt{\kset{\alpha'}}{i}|}}e^{-2\pi i\,l\mu_{\alpha'}(j-\theta_{\alpha'}^{\bar{j}_i}+\theta_{\alpha'}^0)-\theta_{\alpha'}^{\bar{j}_i}+\theta_{\alpha'}^0})\right| \leq \min(\const\,|l|\lS\prt{\dot\Theta}{i}^{-1},2)
\end{align*}
and $f_i=\const\lS^{\beta-1/2}X_i$:
\begin{align*}
\lambda&=\const\cdot\lS^{\onehalf-\beta}& \alpha&=1.
\end{align*}
which yields a bound of $\min(|l|\,C\lS^{1-2\beta}\log\lS,2)$. The remaining term
\eref{e_estimateUpsilonc} can be bounded directly using \eref{e_RiemannSum}, which gives a
bound of $\min(\const\,|l|\lS^{-1},1)$ and concludes the proof of the estimate
\eqref{e_estimateUpsilon1}.  Define
\[
\hat\Psi_{\alpha,2}^{(\alpha',l)}=\hat\Psi_{\alpha',1}^{(l)}\Upsilon_{\alpha'}^{(l)}e^{2\pi
  i l\Phi_{\alpha'}},\] where $\Phi_{\alpha'}$ is a phase to be fixed later. Notice that
by \eref{e_estimateUpsilon1} and since $\beta>\onehalf$ we obtain
\eref{e_mainDecaya}. Summarizing we have:
\[
\Psi_{\alpha,2}(\eta)=\sum_{\alpha'\in
  \asets}{\sum_{l}}\hat\Psi_{\alpha,2}^{(\alpha',l)}e^{-2\pi i l
  (\T(\eta_{\alpha'}^0(\eta^{*}))-\Phi_{\alpha'})}+\|\averaged\|_1 o(\lS\inv).
\]
We now study the exponential term. By definition of $\eta^*$:
\[
\eta_{\alpha'}^0(\eta^*)=\eta^*+2\dot\phi(\bar{\ct{}}_{\alpha'}),
\]
therefore we have
\begin{align*}
\T(\eta_{\alpha'}^0(\eta^*))&=\T(\lS+2\dot\phi(\bar{\ct{}}_{\alpha'}))+(\T(\eta^*)-\T(\lS))\\&+\int_{\T(\lS)}^{\T(\eta^*)}\left(\frac{\T'(\cv{}(\T)+2\dot\phi(\bar{\ct{}}_{\alpha'}))}{\T'(\cv{}(\T))}-1\right)\deh\T.
\end{align*}
By \eref{e_conditionsNu} we conclude that:
\[
\frac{\T'(\cv{}(\T)+2\dot\phi(\bar{\ct{}}_{\alpha'}))}{\T'(\cv{}(\T))}-1=\frac{\T'(\lS+2\dot\phi(\bar{\ct{}}_{\alpha'}))}{\T'(\lS)}-1+\bigo{\lS^{-2}};
\]
define
\[
\we_{\alpha'}=\frac{\T'(\lS+2\dot\phi(\bar{\ct{}}_{\alpha'}))}{\T'(\lS)}-1=\bigo{\lS^{-1}};
\]
notice that $\we_{\alpha'}$ satisfies the bound \eqref{e_upperlowerWe} because of our
choice of $\asets$.
Then by letting
$\dot\phi_{\alpha'}=\T(\lS+2\dot\phi(\bar{\ct{}}_{\alpha'}))+\T(\eta^*)-\T(\lS)-\we_{\alpha'}\cdot\T(\lS)$
(recall that the fractional part of $\T(\eta^*)$ does not depend on $\eta$), we have:
\[
|e^{-2\pi i l (\T(\eta_{\alpha'}^0(\eta^{*}))-\Phi)}-e^{-2\pi i l\we_{\alpha'}\T(\eta)}|
\leq \min(\const\, |l|\lS^{-2+2\beta},2).
\]
By definition of $\hat\Psi_{\alpha,1}^{(l)}$ and $\hat\Psi_{\alpha,2}^{(l)}$ and by estimate \eref{e_estimateUpsilon1} we then have:
\begin{align*}
  &\|\Psi_{\alpha,2}(\eta)-\sum_{\as\in \asets}{\sum_{l}}\hat\Psi_{\alpha,2}^{(\as,l)}e^{-2\pi i l\we_{\as}\T(\eta)}\|\leq \\
  &\leq {\sum_{l}}|\hat\averaged_l|\lS^{-\beta}\cdot (C\lS^{\onehalf-\beta}+\min(|l| C\lS^{1-2\beta}\log\lS+|l| C\lS^{-1},3))\cdot\\
  &\hphantom{\leq \sum_{\as\in \asets}{\sum_{l}}|\hat\averaged_l|\lS^{-\beta}\log\lS}\ \cdot\min(\const\, |l|\lS^{-2+2\beta},2)\leq\\
  &\leq {\sum_{l}}|l|^2|\hat\averaged_l|o(\lS^{-1}).
\end{align*}

Assume now $k\geq 3$; define $\T_{j,\alpha'}(\eta^*)=\T(\eta_\alpha^0(\eta^*))+j+\mu_{\alpha'}(j)$; by Lemma~\ref{l_baseInductionStep} we have $|\T_{j,\alpha'}-\T(\eta_{\alpha'}^j(\eta^*))|<1$, so that first by Lemma~\ref{l_lipshitz} and then by inductive hypothesis we can write:
\begin{align*}
  &\sum_{j\in\Kset{{\alpha'}}}c_{\alpha'}^j\Psi_{\alpha',k-1}(\eta_{\alpha'}^j(\eta^{*}))=\\
  &=\sum_{j\in\Kset{{\alpha'}}}c_{\alpha'}^j\Psi_{\alpha',k-1}(\eta(\T_j))+\|\averaged\|o(\lS^{-1})=\\
  &=\sum_{j\in\Kset{{\alpha'}}}c_{\alpha'}^j\sum_{\as,l}\hat\Psi_{\alpha',k-1}^{(\as,l)}e^{2\pi i\,l\we_\as\T_j}+\|\averaged\|o(\lS^{-1})=\\
  &=\sum_{\as,l}^{\phantom{K_{\alpha'}-1}}\hat\Psi_{\alpha',k-1}^{(\as,l)}e^{-2\pi i l \we_\as Y(\eta_{\alpha'}^0)}\underbrace{\sum_{j\in\Kset{{\alpha'}}}c_{\alpha'}^j e^{-2\pi i\,l\we_\as(j+\mu_{\alpha'}(j))}}_{\Upsilon_{\as,\alpha'}^{(l)}}+
  \|\averaged\|o(\lS^{-1}).
\end{align*}
We claim that
\begin{equation}\label{e_claimUpsilon2}
  |\Upsilon_{\as,\alpha'}^{(l)}|\leq C\lS^{\onehalf-\beta}+\min(C\,|l|\lS^{-1},1).
\end{equation}
Define in the same way as for the case $k=2$ the index set $\kset{\alpha'}\subset
\Kset{\alpha'}$ and fix a partition in subsets $\prt{\kset{\alpha'}}{i}$ such that
$j,j'\in\prt{\kset{\alpha'}}{i}$ if and only if:
\[
\lfloor\we_\as\cdot(j+\mu_{\alpha'}(j)\rfloor = \lfloor\we_\as\cdot(j'+\mu_{\alpha'}(j'))\rfloor.
\]
Define $\prt{\dot\Theta}{i}$, $\prt{c_{\alpha'}}{i}$ and $\bar j_i$ as before. Notice that \eref{e_maxTheta} still holds true, hence so do all estimates on $\prt{c_{\alpha'}}{i}$ and $\prt{\dot\Theta}{i}$. We discard once more the first and last sets $\prt{\kset{\alpha'}}{i}$; for the remaining ones the following estimate holds true:
\[
\we_\as(j+\mu_{\alpha'(j)})=\we_\as(\bar j_i+\mu_{\alpha'}(\bar j_i))+\we_\as(j-\bar j_i+\bigo{1}),
\]
hence, by comparison with a Riemann sum of $\int e^{2\pi i l\theta}\deh\theta$ we obtain the following estimate:
\begin{equation}\label{e_RiemannSum2}
  \sum_j\frac{1}{|\prt{\kset{\alpha'}}i|}e^{-2\pi i l\mu_q\cdot(j+\mu_{\alpha'}(j))}\leq \min(C\,|l|\lS^{-1},1).
\end{equation}
So that now we can estimate:
\begin{subequations}
\begin{align}
  |\Upsilon_{\as,\alpha'}^{(l)}|\leq\sum_i\prt{c_{\alpha'}}{i}&\sum_{j\in\prt{\kset{\alpha'}}{i}}\left|\frac{c_{\alpha'}^j}{\prt{c_{\alpha'}}{i}}-\frac{1}{|\prt{\kset{\alpha'}}{i}|}\right|+\label{e_estimateUpsilon2a}\\+&\left|\sum_{j\in\prt{\kset{\alpha'}}{i}}\frac{1}{|\prt{\kset{\alpha'}}i|}e^{-2\pi i l\we_\as\cdot(j+\mu_{\alpha'}(j))}\right|.\label{e_estimateUpsilon2b}
\end{align}
\end{subequations}
The estimate for \eref{e_estimateUpsilon2a} is the same as for \eref{e_estimateUpsilona}; the estimate for \eref{e_estimateUpsilon2b} is given by \eref{e_RiemannSum2} from which we conclude the proof of \eref{e_claimUpsilon2}.
Define
\[
\hat\Psi_{\alpha,k}^{(\as,l,\alpha')}=\hat\Psi_{\alpha',k-1}^{(\as,l)}\Upsilon_{\as,\alpha'}^{(l)}e^{2\pi i l\Phi_{\alpha'}}
\]
where $\Phi_{\alpha'}$ is a phase to be determined later. Hence we can write:
\[
\Psi_{\alpha,k}(\eta)=\sum_{\as\in \aset}\sum_{\alpha'\in\aset}{\sum_{l\in\integers}}\hat\Psi_{\alpha,k}^{(\as,l,\alpha')}e^{-2\pi i l \we_\as(\T(\eta_{\alpha'}^0(\eta^*))-\Phi_{\alpha'})}
\]
We now need to study the oscillating term $e^{-2\pi i
  l\we_\as\T(\eta_{\alpha'}^{0}(\eta^*))}$; we now let
$\Phi_{\alpha'}=\T(\lS+2\dot\phi(\bar x_{\alpha'}))-\T(\lS)-\we_\alpha'\cdot\T(\lS)$ so
that $\T(\eta_{\alpha'}^0(\eta^*))=\Phi_{\alpha'}+\T(\eta^*)+\bigo{\lS^{-1+2\beta}}$,
which implies:
\[
e^{-2\pi i l\we_\as\T(\eta_{\alpha'}^{0}(\eta^*))}=e^{-2\pi i l\we_\as\T(\eta^*)}+\min(|l|\bigo{\lS^{-2+2\beta}},1),
\]
whose main term does not depend on $\alpha'$.  We can thus define
$\hat\Psi_{\alpha,k}^{(\as,l)}=\sum_{\alpha'\in\aset}\hat\Psi_{\alpha,k}^{(\as,l,\alpha')}$
which, by \eref{e_claimUpsilon2}, implies \eref{e_mainDecayb} and therefore
\eref{e_explicitDecay} . We then have:
\begin{align*}
  &\|\Psi_{\alpha,k}(\eta)-\sum_{\as\in \asets}{\sum_{l}}\hat\Psi_{\alpha,k}^{(\as,l)}e^{-2\pi i l\we_\as\T(\eta)}\|\leq \\
  &\leq {\sum_{l}}|l||\hat\averaged_l| (C\lS^{-\beta-(k-1)(\beta-\onehalf)}\log\lS+\min(|l|\lS^{-1},1))\cdot\\
  &\hphantom{\leq \sum_{q\in Q}{\sum_{l}}|\hat\averaged_l|\lS^{-\beta}\log\lS}\ \cdot\min(\const\, |l|\lS^{-2+2\beta},2)\leq\\
  &\leq {\sum_{l}}|l|^2|\hat\averaged_l|o(\lS^{-1})
\end{align*}
\end{proof}
\begin{rmk}
The proof of Lemma~\ref{l_periodicity} clarifies why the cancellation
argument fails for $\gamma\leq2$; consider for instance the case $\gamma=2$, then, in the
formula for $\Upsilon$ we have a sum of $\bigo{\lS}$ terms whose phases differ by
$\bigo{\lS^{-1}}$; hence, incomplete boundary collections of indices will contribute with
some fixed proportion, and we cannot expect to iterate efficiently the cancellation
scheme.

As we mentioned in the introductory section, we performed simple numerical computations
for the functions $\Psi_{\alpha,k}$ for various values of $\gamma$; if $\gamma$ is far
enough away from $2$, the outcome of such computations well agree with \eqref{e_equi}
apart from the $o(\lS^{-1})$ term, which appears to be a technical byproduct of our
techniques and, in principle, could be avoided by improving some of the estimates. As $\gamma\to
2^{+}$, the asymptotic behavior \eqref{e_equi} appears to dominate only for larger and
larger values of $\lS$; it is thus increasingly delicate to obtain sensible quantitative
results in this region, nevertheless the asymptotics \eqref{e_equi} still appears to be
the best fit.
\end{rmk}

%% file: reduction.tex
\input{definitions.p}
\newcommand{\smallo}[1]{o(#1)}
\section{Comparison with a biased random walk}\label{s_reduction}
In this section we describe a procedure which allows to compare the dynamics on a standard
pair with a one-dimensional biased random walk; the comparison argument is the crucial
ingredient for the proof of Lemma~\ref{l_criticalTimeEstimate}. All arguments given in
this section are adapted from the analogous ones explained in \cite{Dima}; the only
possibly non-trivial adaptation is the proof of proposition \ref{l_final}.

Let us denote by $\bar\ell$ the standard pair appearing in the statement of
Lemma~\ref{l_criticalTimeEstimate}: we will call $\bar\ell$ the \emph{master} standard
pair. For $k\in\integers$ define $R_k=2^k\V_{\bar\ell}$; we say that a standard pair
$\ell$ is \emph{close} to $R_k$ if
\[
\Gamma_{\ell}\subset\torus^1\times[R_k-2A\nu,R_k+2A\nu],
\]
where $\nu$ is the one given by Lemma~\ref{l_equidistribution}.  We say that a standard pair
$\ell$ is \emph{compatible} with $R_k$ if
\[
\Gamma_{\ell}\subset\torus^1\times[R_{k-1},R_{k+1}].
\]
\begin{definition}
Let $\ell$ be a standard pair; following definition \ref{d_criticalTime} we introduce the function
\begin{align*}
\tau_\ell^{[k]}&:\Gamma_\ell\to\naturals\cup\{\infty\}
\end{align*}
given by the following recursive definition:
if $\ell$ is not compatible with $R_k$ we let $\tau_\ell^{[k]}\equiv 0$. Otherwise let $p\in\Gamma_\ell$, then by item (b) of lemma \ref{l_invariance} we have three possibilities:
\begin{itemize}
\item $Fp$ belongs to a standard pair $\ell'=(\Gamma_{\ell'},\rho_{\ell'})$: we then let $\tau_\ell^{[k]}(p)=\tau_{\ell'}^{[k]}(Fp)+1$;
\item $Fp$ belongs to a stand-by pair, hence $F^2p$ belongs to a standard pair $\ell''$: we then let $\tau_\ell^{[k]}(p)=\tau_{\ell''}^{[k]}(F^2p)+2$;
\item otherwise we let $\tau_\ell^{[k]}(p)=0$.
\end{itemize}
Notice that by definition we have $\tau_\ell^{[k]}\leq\tau_\ell$.
\end{definition}
\begin{definition}
  Let $\ell$ be a standard pair close to $R_k$; we then define a function $\xi_\ell^{[k]}:\Gamma_\ell\to\{-1,+1\}$ in the following way:
\[
\xi_\ell^{[k]}(p)=
\begin{cases}
+1&\textrm{if }\tau_\ell^{[k]}(p)<\tau_\ell(p)\textrm{ and }F^{\tau_\ell^{[k]}(p)}(p)\textrm{ belongs to $\ell'$ close to }R_{k+1};\\
-1&\textrm{ otherwise.}
\end{cases}
\]
\end{definition}
The main technical result of this section, which will be used to prove lemma \ref{l_criticalTimeEstimate} is given by the following
\begin{prp}\label{l_final}
Let $\ell$ be a standard pair; then if $\gamma>2$:
\begin{enumerate}[(a)]
\item there exists $0<\vartheta<1$ such that $\prob_\ell(\tau_\ell^{[k]}\geq s)\leq \const\vartheta^{s\V^{-2}_\ell}$;
\item if $\ell$ is close to $R_k$, then we have $\prob_\ell\left(\xi_\ell^{[k]}=-1\right) \geq 0.6$;
\end{enumerate}
\end{prp}
We now show how proposition \ref{l_final} implies lemma \ref{l_criticalTimeEstimate} and
postpone its proof to the end of the current section.  Define two sequence of functions on
the master standard pair $\bar\ell$:
\begin{align*}
  \tau_k&:\Gamma_{\bar\ell}\to\naturals\cup\{\infty\}&\chi_k&:\Gamma_{\bar\ell}\to\integers,
\end{align*}
such that if $\tau_k(p)<\tau(p)$, then $F^{\tau_k(p)}p$ belongs to a standard pair $\ell'$ which is close to $R_{\chi_n(p)}$. We proceed by induction: let $\tau_0\equiv 0$ and $\chi_0\equiv 0$; assume we already defined $\tau_k$ and $\chi_k$: then if $\tau_k(p)=\tau(p)$ we set:
\begin{align*}
\tau_{k+1}(p)&\defeq\tau_{k}(p)&\chi_{k+1}(p)&\defeq\chi_{k}(p)-1.
\end{align*}
Otherwise, by definition $F^{\tau_{k}(p)}(p)=p'$ belongs to some standard pair $\ell'$ close to $R_{\chi_k(p)}$; we then define:
\begin{align*}
  \tau_{k+1}(p)&\defeq \tau_{k}(p)+\tau_{\ell'}^{[\chi_k(p)]}(p')&
  \chi_{k+1}(p)&\defeq \chi_{k}(p)+\xi_{\ell'}^{[\chi_k(p)]}(p').
\end{align*}
The proof of lemma \ref{l_criticalTimeEstimate} now follows from the same argument which
has been used in \cite{Dima} to prove the corresponding estimate (23); we sketch the
argument here and refer the reader to the said reference for the detailed proofs, which
could be repeated verbatim in our situation. The crucial observation is that item (b) of
proposition \ref{l_final} implies that we can compare the dynamics of our system outside
$\critical2$ with a biased random walk moving up with probability $0.4$ and moving down
with probability $0.6$; by this comparison, and by item (a) of proposition \ref{l_final}
we obtain that, almost every point on a standard pair will visit the $\critical2$ (which
includes the region $\{y\leq y_*\}$) in finite time, that is the statement of lemma
\ref{l_criticalTimeEstimate}.
We are now left concluding with the
\begin{proof}[Proof of proposition \ref{l_final}]
  First of all notice that if $\ell$ is not compatible with $R_k$, then item (a) is trivially satisfied; we therefore assume that $\ell$ is compatible with $R_k$. Define the following function on $\ell$:
  \[
  \zeta_n(\ct{0})\defeq\dot\phi(\ct{n+\nu+1})1_{\tau_\ell^{[k]}\geq n}1_{\tau_\ell\geq n+\nu},
  \]
  We claim that the following expressions hold
  \begin{subequations}\label{e_crucialEstimates}
    \begin{align}
      \expec\ell{\zeta_n}&=\prob_\ell(\tau_\ell^{[k]}\geq n)\smallo{\V_\ell^{-1}}\label{e_crucialEstimateSingle}\\
      \expec\ell{\left(\sum_{n=1}^N\zeta_n\right)^2}&\geq N\cdot2A^2\prob_\ell(\tau_\ell^{[k]}\geq N)+\smallo{N}.\label{e_crucialEstimateSum}
    \end{align}
  \end{subequations}
  We now show that equations~\eqref{e_crucialEstimates} imply proposition~\ref{l_final}: in fact by definition of $\zeta_n$ we have:
  \[
  \left\|\sum_{n=0}^N\zeta_n\right\|<3\V_\ell+2 A(\nu+1),
  \]
  hence
  \[
  \expec\ell{\left(\sum_{n=1}^N\zeta_n\right)^2}\leq 9\V_\ell^2+\smallo{\V_\ell^2}
  \]
  which, by \eqref{e_crucialEstimateSum} implies:
  \[
  N\cdot2A^2\prob_\ell(\tau_\ell^{[k]}\geq N) \leq 9\V_\ell^2+\smallo{\V_\ell^2}.
  \]
  Taking $N=L\V_\ell^2$ for large enough $L$ and dividing the previous expression by $2A^2L\V_\ell^2$ we thus obtain:
  \begin{equation}\label{e_thetaNoniterated}
    \prob_\ell(\tau_\ell^{[k]}\geq L\V_\ell^2) \leq \vartheta
  \end{equation}
for some $0<\vartheta<1$. The previous expression implies item (a) by the following argument:
for $m\in\naturals$ let
  \[
  \mathcal{N}_m=[\lceil m\cdot L\V_\ell^2\rceil,\lfloor (m+1)\cdot L\V_\ell^2\rfloor].
  \]
  thus $s\in\mathcal{N}_m$ for some $m\in\naturals$. Consequently, for any $m'\leq m$ we have that $F^{m' L\V_\ell^2}\{\tau_\ell^{[k]}\geq s\}$ can be decomposed in standard pairs and on each one we can apply \eqref{e_thetaNoniterated} since by definition such standard pairs are still compatible with $R_k$. Hence, by induction we obtain
  \[
  \prob_\ell(\tau_\ell^{[k]}\geq s) \leq \vartheta^m
  \]
  which implies item (a). To prove item (b), use \eqref{e_crucialEstimateSingle} and write:
\begin{align*}
  \expec{\ell}{\sum_{n=1}^\infty\zeta_n}&\leq\sum_n \prob_\ell\left(\tau_\ell^{[k]}\geq n\right)\smallo{\V^{-1}}\\&\leq \expec{\ell}{\tau_\ell^{[k]}}\smallo{\V^{-1}}.
\end{align*}
By item (a) we know that $\expec{\ell}{\tau_\ell^{[k]}}=\bigo{\V_\ell^2}$; 
on the other hand:
\[
\expec{\ell}{\sum_{n=1}^\infty\zeta_n} = \V\cdot \prob_\ell\left(\xi_\ell^{[k]}=+1\right) +\lambda \V\cdot  \prob_\ell\left(\xi_\ell^{[k]}=-1\right) +\bigo{1},
\]
where $\lambda\in(-1/2,1)$; dividing by $\V$ we obtain:
\[
\prob_\ell\left(\xi_\ell^{[k]}=+1\right) +\lambda \prob_\ell\left(\xi_\ell^{[k]}=-1\right)=o(1)
\]
which implies:
\[
\prob_\ell\left(\xi_\ell^{[k]}=-1\right) = \frac{1}{1-\lambda}(1+o(1))>0.6
\]
that is item (b).

We now only need to prove equations \eqref{e_crucialEstimates}: By applying $n$ times lemma \ref{l_invariance} and discarding those pairs that do not satisfy $\tau_\ell^{[k]}\geq n$ we obtain the following decomposition:
\begin{equation}\label{e_decompositionForN}
F^n\ell=\bigcup_j\ell'_j\cup\bigcup_n\tilde\ell_n\cup\{\tau_\ell^{[k]}<n\}
\end{equation}
and
\[
F\bigcup_l\tilde\ell_l=\bigcup_j\ell''_j.
\]
where $\ell'_j$ and $\ell''_j$ are standard pairs. Let $c'_j=\prob_\ell(F^{-n}\Gamma_{\ell'_j})$ and $c''_j=\prob_\ell(F^{-n-1}\Gamma_{\ell''_j})$; then, by definition we have
\[
\sum_{j}c'_j + \sum_{j}c''_j=\prob_\ell(\tau_\ell^{[k]}\geq n).
\]
Thus, using lemma \ref{l_equidistribution} we have:
\begin{align*}
  \expec{\ell}{\zeta_{n}}&
  =\sum_j c'_j \expec{\ell'_j} {\dot\phi\circ F^{\nu+1}  1_{\tau_\ell\geq \nu}}
  +\sum_j c''_j\expec{\ell''_j}{\dot\phi\circ F^{\nu}1_{\tau_\ell\geq \nu-1}}\\
  &\leq \sum_j c'_j\smallo{\V_\ell^{-1}} +\sum_j c''_{j}\smallo{\V_\ell^{-1}}
\end{align*}
which implies \eqref{e_crucialEstimateSingle}. We now apply the same argument to the functions $\zeta_n^2$ and obtain:
\begin{align}
  \expec{\ell}{\zeta_{n}^2}&
  =\sum_j c'_j \expec{\ell'_j} {\dot\phi^2\circ F^{\nu+1}  1_{\tau_\ell\geq \nu}}
  +\sum_j c''_j\expec{\ell''_j}{\dot\phi^2\circ F^{\nu}1_{\tau_\ell\geq \nu-1}};\notag
  \intertext{whence, extracting the average value $\dot\phi^2=(\dot\phi^2-2A^2)+2A^2$ we obtain:}
  \expec{\ell}{\zeta_{n}^2}&= (2A^2+\bigo{\V^{-\beta}_\ell})\prob_\ell(\tau_\ell^{[k]}\geq n)+\smallo{\V_\ell^{-1}}\label{e2_final}
\end{align}
Next, we claim that, for $m\in\naturals$ we have:
\begin{equation}\label{e4_final}
\expec{\ell}{\zeta_m\sum_{i=0}^{m-1}\zeta_i}=\smallo{1}.
\end{equation}
In fact, applying decomposition \eqref{e_decompositionForN} to $F^m$, we obtain:
\begin{align*}
\expec{\ell}{\zeta_m\sum_{i=0}^{m-1}\zeta_i}&=\sum_jc'_j\expec{\ell'_j}{\zeta_{0}\sum_{i=-m}^{-1}\zeta_{i}}+\\&\quad+\sum_jc''_{j}\expec{\ell''_{j}}{\zeta_{-1}\sum_{i=-m-1}^{-2}\zeta_i}.
\end{align*}
We estimate separately the contribution of each standard pair in each of the two terms on the right hand side: fix $\ell'_j$ and define on $\Gamma_{\ell'_j}$ the function $B=\sum_{i=-m}^{-1}\zeta_i$; we want to prove that $\expec{\ell}{B\zeta_0}=\smallo{1}$. Fix $p>\nu+1$ to be determined later and decompose $B=B_1+B_2$ as follows:
\begin{align*}
B_1 &= \sum_{i=-m}^{-p}\zeta_i&
B_2 &= \sum_{i=-p+1}^{-1}\zeta_i.
\end{align*}
where  if $m <p$ we assume conventionally that $B_1=0$ and $B_2=B$. By definition of $\tau^{[k]}_\ell$ we have $\|B_1\|_{\infty}\leq 3\V+2(\nu+1) A\leq 4\V$; moreover $B_1$ depends only on $\ct{i}$ with $i<-(p-\nu-1)$, hence, by item (a) of lemma \ref{l_invariance}:
\[
\|\dot B_1\|_{\infty} = \bigo{\V^{-(p-\nu-1)\beta}}
\]
To estimate the contribution of $B_1$, we write $B_1=\bar B_1+\tilde B_1$, where $\bar B_1$ is the constant part of $B_1$; then $\|\tilde B_1\|_{\infty} = \bigo{\V^{-(p-\nu-1)\beta}}$ and we can write, using corollary~\ref{c_lazyInduction} and theorem \ref{l_equidistribution} and requiring $p$ to be large enough, that:
\begin{align*}
\expec{\ell'_{j}}{B_1 \zeta_0}& = \smallo{1}.
\end{align*}
Consider now the remaining term $B_2$; by definition we have $\|B_2\|_{\infty}\leq 2A(p-1)$; moreover, if $\ct{\nu}$ belongs to a standard pair we have:
\[
\left\|\de{B_2}{\ct{\nu}}\right\|=\bigo{1},
\]
so that we obtain by corollary~\ref{c_lazyInduction}:
\begin{align*}
  \expec{\ell'_{j}}{B_2 \zeta_0}&=\smallo{\V^{-1}}.
\end{align*}
The terms involving $\ell''_j$ can be treated analogously and, by linearity of the expectation, we can conclude that \ref{e4_final} holds. Finally, using \eref{e2_final} and \eref{e4_final} we obtain:
\begin{align*}
\expec{\ell}{\left(\sum_{i=0}^N\zeta_i\right)^2}&=\sum_{i=0}^N \left(2A^2\prob_\ell\left(\tau_\ell\geq i\right)+ o(1)\right)\\
&\geq N\cdot 2A^2\prob_\ell\left(\tau_\ell\geq N\right) + N\cdot o(1).
\end{align*}
which concludes the proof.
\end{proof}

%% file: criticalsets.estimates.tex
\input{definitions.p}
\section{Proof of Lemma~\ref{l_lebesgueCritical}}\label{s_proofLemmaLebesgue}
In order to prove items ($a_1$) and ($a_2$), it suffices to check that, for any standard
curve $\Gamma$, the estimates
\begin{align}\label{e_estimateDerivativeSlope}
  \left.\frac{\deh}{\deh x}\tslope{1}\right|_{\Gamma}&=\bigo{1}&
  \left.\frac{\deh}{\deh x}\tslope{1}(\tslope{1}\circ F)\right|_{\Gamma}&=\bigo{1}
\end{align}
hold in the specified neighborhood; then the proof trivially follows from the definition
of $\hcritical1$ and $\hcritical2$.  In turn \eqref{e_estimateDerivativeSlope} easily
follows from the definition of $\tslope1$; in fact recall that
\begin{align*}
\partial_{\ct{}}\tslope{1}&=\bigo{1}& \partial_{\cv{}}\tslope{1}&=\bigo{\cv{}\inv};
\end{align*}
thus
\[
  \left.\frac{\deh}{\deh x}\tslope{1}\right|_{\Gamma}=\partial_{\ct{}}\tslope1+\T'\slope\Gamma\cdot\partial_{\cv{}}\tslope1=\bigo1;
  \]
  which concludes the proof of item ($a_1$).  Similarly, notice that:
\[
\left.\frac{\deh}{\deh x}\tslope{1}(\tslope{1}\circ F)\right|_{\Gamma}=
\left.\frac{\deh}{\deh x}\tslope{1}\right|_{\Gamma}\cdot \tslope{1}\circ F+
  \tslope{1}\cdot\left.\frac{\deh}{\deh x}\tslope{1}\circ F\right|_{\Gamma};
\]
the first term of the right hand side is $\bigo1$ by the previous argument, on the other
hand the second term can be bounded as follows:
\[
\tslope{1}\cdot\left.\frac{\deh}{\deh x}\tslope{1}\circ F\right|_{\Gamma}=
\tslope{1}\cdot\left(\partial_{\ct{}}\tslope1\left.\frac{\deh\ct1}{\deh\ct{}}\right|_{\Gamma}+
\partial_{\cv{}}\tslope1\left.\frac{\deh\cv1}{\deh\ct{}}\right|_{\Gamma}\right).
\]
By definition
$\left.\frac{\deh\ct1}{\deh\ct{}}\right|_{\Gamma}=\T'\tslope{\Gamma}=\bigo{\cv{}^{\beta}}$,
which implies that the first term is $\bigo1$; moreover, by \eqref{e_firstIteration} we
have that
$\left.\frac{\deh\cv1}{\deh\ct{}}\right|_{\Gamma} =
\bigo{\T'}\left.\frac{\deh\ct1}{\deh\ct{}}\right|_{\Gamma}$,
from which we conclude the proof of item ($a_2$).  The proof of item $(b_1)$ is simple,
since, we have that $\left.\frac{\deh}{\deh x}\tslope{1}\right|_{\Gamma}$ is bounded away
from zero in $\Gamma\cap\hcritical1$.  Concerning item $(b_2)$, it is not difficult to
see, by direct inspection (see Figure~\ref{f_criticals}) that the number of connected
components of $\Gamma\cap\hcritical2$ is bounded by $N_\Gamma+4$, where $N_\Gamma$ is the
number of intersections of $F(\Gamma\cap\hcritical1)$ with the vertical line
$\{\ct{}=0\}$; it is thus sufficient to prove that the number of such intersections is
uniformly bounded in $\V_\Gamma$.  This, however, is simple to achieve since, by
definition, $\Gamma$ has a quadratic critical point inside $\hcritical1$, and its
curvature is bounded from above by $4A\T'(\V_\Gamma)$, therefore its image will intersect
the said vertical line in at most $2\cdot 4A\T'(\V_\Gamma)\hat
K_1^2\V_\Gamma^{-2\beta}=\bigo1$ points, which proves item ($b_2$).  To prove item
($c_1$), notice that, by definition:
\[
|\tslope1(\coo{0})|=|2\ddot\phi(\ct{0})+1/\T'(\cv{-1})+1/\T'(\cv{0})|;
\]
using \eqref{e_confuse} we can write:
\begin{align*}
\hcritical1&\subset\{(\coo{})\st|2\ddot\phi(\ct{})|<2K_1\T'(\cv{})^{-1/2} \}\\
&\subset\{|\ct{}|<{\rm Const}\cdot \T'(\cv{})^{-1/2}\}\cup\{|\ct{}-1/2|<{\rm Const}\cdot \T'(\cv{})^{-1/2}\}.
\end{align*}
Denote the two sets that appear in the last expression by $\critical{1}^{(0)}$ and
$\critical{1}^{(1)}$ respectively; the Lebesgue measure of $\critical{1}^{(i)}$ is finite
if the function $\T'^{-1/2}$ is integrable at $\infty$, i.e.\ if $\beta>1$, that is, if
$\gamma>3$.  In the same way we can obtain a lower bound, so that if $\gamma\leq3$ then
$\leb{\critical{1}}=\infty$.  Similarly, in order to prove item ($c_2$), define, for
$i\in\{0,1\}$ and $n\in\naturals$:
\[
\hcritical{2}^{(i,n)}= \hcritical{2} \cap \hcritical{1}^{(i)} \cap \{(\coo{})\st \ct{}+\T(\cv{})\in[n/2,(n+1)/2]\};
\]
also let $\V_n=\inf_{(\coo{})\in\hcritical{2}^{(i,n)}} \cv{}\sim n^{1/\gamma}$.  Then, for
each $\hcritical{2}^{(i,n)}$, consider the following decomposition (see also figure
\ref{f_C2}):
\begin{align*}
\hcritical{2}'^{(i,n)}&=\{(\coo{0})\in \hcritical{2}^{(i,n)}\st |\tilde
h_1(\coo{1})|<(\hat K_2/\hat K_1)\T'(\cv{1})^{-1/2}\}\\
\hcritical{2}''^{(i,n)}&=\{(\coo{0})\in \hcritical{2}^{(i,n)}\st |\tilde h_1(\coo{1})|<A\}\setminus \hcritical{2}'^{(i,n)}\\
\hcritical{2}'''^{(i,n)}&= \hcritical{2}^{(i,n)}\setminus (\hcritical{2}'^{(i,n)} \cup \hcritical{2}''^{(i,n)}).
\end{align*}
\begin{figure}[!b]
  \center
  \def\svgwidth{4cm}
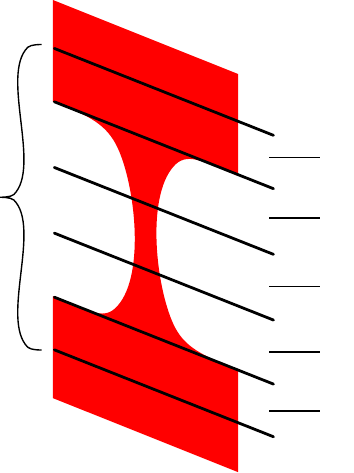
\caption{Decomposition of $\hcritical{2}^{(i,n)}=\hcritical{2}'^{(i,n)}\cup\hcritical{2}''^{(i,n)}\cup\hcritical{2}'''^{(i,n)}$.}
\label{f_C2}
\end{figure}
First consider $(\coo{})\in \hcritical{2}'''^{(i,n)}$; by definition we have:
\[
|\tilde h_1(\coo{})|<\frac{2\hat K_2}{A\T'(\cv{})}
\]
which is a bound for $\ct{}$ of order $\bigo{\cv{}^{-2\beta}}$, so that:
\[
\leb{\hcritical{2}'''^{(i,n)}}\leq\const \V_n^{-4\beta}.
\]
The measure of $\hcritical{2}'^{(i,n)}$ and $\hcritical{2}''^{(i,n)}$ can be estimated using the following change of variables:
\[
(\coo{0})\mapsto(\xi,\eta)=(\tilde h_1(\coo{0}),\tilde h_1(\coo{1}));
\]
this map is an invertible diffeomorphism and its Jacobian determinant is of order
$\T'(\V_n)$; for convenience denote $\T'_n=\T'(\V_n)$.  Therefore, for
$\hcritical{2}'^{(i,n)}$ we obtain:
\[
\leb{\hcritical{2}'^{(i,n)}}\leq\frac{2}{\T'_n}\int_{-2(\hat K_2/\hat
  K_1){\T'_n}^{-1/2}}^{+2(\hat K_2/\hat K_1){\T'_n}^{-1/2}}\int_{-2\hat K_1
  {\T'_n}^{-1/2}}^{+2\hat K_1 {\T'_n}^{-1/2}}\deh\xi\deh\eta=\bigo{\V_n^{-4\beta}}
\]
and for $\hcritical{2}''^{(i,n)}$:
\[
\leb{\hcritical{2}''^{(i,n)}}\leq\frac{1}{\T'_n}\int_{\frac{1}{2}(\hat K_2/\hat K_1) {\T'_n}^{-1/2}}^{A}\int_{-2\hat K_2/\eta \T'_n}^{+2\hat K_2/\eta \T'_n}\deh\xi\deh\eta=\bigo{\V_n^{-4\beta}\log \V_n}.
\]
Therefore we finally have:
\[
\leb{\hcritical{2}^{(i,n)}}\leq\const\V_n^{-4\beta}\log \V_n
\]
and summing over $i$ and $n$ we obtain 
\[
\leb{\hcritical{2}}<\infty\ {\rm if}\ \sum_nn^{-\frac{4\beta}{\gamma}}\log n<\infty,
\]
where the series converges if $\beta>1/2$, that is, $\gamma>2$; to conclude, notice that,
by the argument used to prove item ($c_1$), if $\gamma\leq 2$, then $\corecritical{2}$ has
infinite measure, which implies that the same is true for $\hcritical2$ and concludes the
proof of the lemma.  \qed

%% file: C2.pdf_tex
\begingroup%
  \makeatletter%
  \providecommand\color[2][]{%
    \errmessage{(Inkscape) Color is used for the text in Inkscape, but the package 'color.sty' is not loaded}%
    \renewcommand\color[2][]{}%
  }%
  \providecommand\transparent[1]{%
    \errmessage{(Inkscape) Transparency is used (non-zero) for the text in Inkscape, but the package 'transparent.sty' is not loaded}%
    \renewcommand\transparent[1]{}%
  }%
  \providecommand\rotatebox[2]{#2}%
  \ifx\svgwidth\undefined%
    \setlength{\unitlength}{101.72235718bp}%
    \ifx\svgscale\undefined%
      \relax%
    \else%
      \setlength{\unitlength}{\unitlength * \real{\svgscale}}%
    \fi%
  \else%
    \setlength{\unitlength}{\svgwidth}%
  \fi%
  \global\let\svgwidth\undefined%
  \global\let\svgscale\undefined%
  \makeatother%
  \begin{picture}(1,1.33697256)%
    \put(0,0){\includegraphics[width=\unitlength]{C2.pdf}}%
    \put(-0.02136853,0.76356065){\color[rgb]{0,0,0}\makebox(0,0)[rb]{\smash{$\hcritical2^{(i,n)}$}}}%
    \put(0.92502648,0.15092745){\color[rgb]{0,0,0}\makebox(0,0)[lb]{\smash{$\hcritical2'^{(i,n)}$}}}%
    \put(0.92502648,0.86866691){\color[rgb]{0,0,0}\makebox(0,0)[lb]{\smash{$\hcritical2'^{(i,n)}$}}}%
    \put(0.92502648,0.69747332){\color[rgb]{0,0,0}\makebox(0,0)[lb]{\smash{$\hcritical2''^{(i,n)}$}}}%
    \put(0.92502648,0.31814886){\color[rgb]{0,0,0}\makebox(0,0)[lb]{\smash{$\hcritical2''^{(i,n)}$}}}%
    \put(0.92502648,0.50276597){\color[rgb]{0,0,0}\makebox(0,0)[lb]{\smash{$\hcritical2'''^{(i,n)}$}}}%
  \end{picture}%
\endgroup%

%% file: bibliography.tex
\bibliographystyle{abbrv}
\bibliography{twoiterates}